\documentclass[12pt]{amsart}

\usepackage{graphicx}
\usepackage{amsmath, comment}
\usepackage{amscd}
\usepackage{amsfonts}
\usepackage{amssymb}
\usepackage{xcolor}
\usepackage[normalem]{ulem}
\definecolor{MyDarkGreen}{rgb}{0.0,0.5,0.0}

\usepackage[
  hypertexnames=false,
  colorlinks,
  citecolor=MyDarkGreen,
  linkcolor=blue,
  urlcolor=black,
  pagebackref
]{hyperref}
\usepackage{mathtools}
\usepackage[all]{xy}
\usepackage{tikz-cd}
\usepackage[capitalise]{cleveref}
\usepackage{enumitem}
\usepackage{xparse} 

\usepackage[lining,tabular,scale=.9]{FiraSans}
\usepackage[tt=false,semibold]{libertine}
\usepackage[libertine,timesmathacc,smallerops]{newtxmath}
\usepackage[lining,scaled=.8]{FiraMono}
\usepackage[cal=boondox,frak=euler]{mathalpha}

\numberwithin{equation}{section}
\newtheorem{theorem}{Theorem}[section]
\newtheorem{lemma}[theorem]{Lemma}

\newtheorem{proposition}[theorem]{Proposition}

\newtheorem{claim}[theorem]{Claim}
\newtheorem{corollary}[theorem]{Corollary}

\newtheorem{question}[theorem]{Question}

\theoremstyle{remark}

\newtheorem{remark}[theorem]{Remark}
\newtheorem{definition}[theorem]{Definition}

\newcommand{\ol}{\overline}

\newcommand{\eps}{\varepsilon}
\newcommand{\vphi}{\varphi}

\newcommand{\Ad}{\operatorname{Ad}}
\newcommand{\im}{\operatorname{im}}

\newcommand{\Gal}{\operatorname{Gal}}
\newcommand{\alg}{\operatorname{alg}}
\newcommand{\Aut}{\operatorname{Aut}}

\newcommand{\End}{\operatorname{End}}
\newcommand{\PGL}{\operatorname{PGL}}

\newcommand{\wt}{\operatorname{wt}}

\newcommand{\Br}{\operatorname{Br}}

\newcommand{\SL}{\operatorname{SL}}

\newcommand{\SO}{\operatorname{SO}}
\newcommand{\Spin}{\operatorname{Spin}}
\newcommand{\HSpin}{\operatorname{HSpin}}

\newcommand{\diag}{\operatorname{diag}}

\newcommand{\ind}{\operatorname{ind}}

\newcommand{\Spec}{\operatorname{Spec}}
\newcommand{\Span}{\operatorname{Span}}

\newcommand{\rank}{\operatorname{rank}}

\newcommand{\Char}{\operatorname{char}}

\newcommand{\lin}{\operatorname{lin}}

\newcommand{\on}{\operatorname}

\newcommand{\depth}{\operatorname{depth}}
\newcommand{\Glin}{G^{\mathrm{sm}, \, \lin}}

\newcommand{\bF}{\mathbb{F}}

\newcommand{\bZ}{\mathbb{Z}}

\begin{document}
\author{Danny Ofek}   
\address{Department of Mathematics\\
	University of British Columbia\\
	Vancouver, BC V6T 1Z2\\Canada}
    \email{dannyofe@math.ubc.ca}
\thanks{Danny Ofek was supported by the Natural Sciences and Engineering Research Council of Canada (NSERC), (reference number 601234).}

\author{Zinovy Reichstein}
\address{Department of Mathematics\\
	University of British Columbia\\
	Vancouver, BC V6T 1Z2\\Canada}
 \email{reichst@math.ubc.ca}
\thanks{{Zinovy Reichstein was partially supported by 
        an Individual Discovery Grant RGPIN-2023-03353 from the
	Natural Sciences and Engineering Research Council of Canada.}
}

\author{Federico Scavia}
\address{CNRS\\
	Institut Galil\'ee\\
	Universit\'e Sorbonne Paris Nord\\
	99 avenue Jean-Baptiste Cl\'ement, 93430\\ 
	Villetaneuse, France}
\email{scavia@math.univ-paris13.fr}


%
%


\title{Finite abelian subgroups of algebraic groups}

\begin{abstract}  Let $k$ be an algebraically closed field, and let $G$ be an algebraic $k$-group. We study finite abelian $k$-subgroups $A \subset G$ whose order is not divisible by the characteristic of $k$. This is a classical topic in the theory of algebraic groups going back to the work of Borel in the early 1960s. 
We sharpen previously known results on the structure of $A$. 
In particular, we show that there exists a maximal torus $T$ of $G$ such that the index $[A: (A \cap T)]$ divides the Grothendieck torsion index $t(G)$. We also show that there exists a maximal torus $T$ such that the quotient group $A/(A \cap T)$ is ``small'' in a suitable sense. As applications of these results, we (i) give a positive answer to a question of Totaro for $G$-torsors over 
fields $k_r = k(\!(t_1)\!)(\!(t_2)\!) \ldots (\!(t_r)\!)$ of iterated Laurent series, (ii) prove a variant of the ``hypoth\`ese optimiste'' of Tits about splitting fields of $E_8$-torsors, and (iii) show that certain torsors over $k_r$ cannot be split by the function field of a genus $1$ curve.
\end{abstract}
 
\maketitle
\tableofcontents

\section{Introduction} 

Let $k$ be an algebraically closed field, and let $G$ be an algebraic $k$-group, that is, a group scheme of finite type over $k$. In this paper we study the finite abelian $k$-subgroups of $G$. A maximal torus of $G$ contains many abelian $k$-subgroups, called \emph{toral}. 
We are interested in ``how far'' an arbitrary finite abelian $k$-subgroup $A$ of $G$ is from being toral. This is a classical topic in the theory of algebraic groups, going back to the work of Borel~\cite{borel-tohoku}. Borel studied the closely related setting where $G$ is replaced by a connected compact Lie group $C$. A prime number $p$ is called a torsion prime of $C$ if the integral cohomology ring $H^*(C, \mathbb Z)$ has non-trivial $p$-torsion elements. Borel showed that $p$ is a torsion prime for $C$ if and only if $H^*(BC, \mathbb Z)$ has $p$-torsion if and only if $C$ has a non-toral elementary abelian $p$-group; see~\cite[Theorem 4.5]{borel-tohoku}. 

\subsection{The Grothendieck torsion index}

Recall that, for a field $K$ and a $K$-scheme of finite type $X$, the index $\ind(X)$ is the greatest common divisor of the degrees of the closed points of $X$ over $K$, or equivalently of the degrees $[L:K]$ of all finite field extensions $L/K$ such that $X(L) \neq \emptyset$.

Suppose that $G$ is an affine algebraic group defined over $k$. The \emph{Grothendieck torsion index} $t(G)$ of the algebraic $k$-group $G$ is defined as follows: 
\begin{equation}\label{eq:t}
    t(G) \coloneqq \mathrm{lcm}\{\ind(X)\mid \text{$K/k$ a field extension, $X$ a $G$-torsor over $K$}\}.
\end{equation} 
Grothendieck~\cite{grothendieck-special} showed that $t(G)$ is a well-defined positive integer. Moreover, if $G$ is a complex reductive group and $C$ is its maximal compact subgroup, then the prime factors of $t(G)$ are exactly the torsion primes of $C$. For a more detailed discussion of the Grothendieck torsion index and torsion primes, see Section~\ref{sect.prelim}. 

For a finite abelian $k$-subgroup $A$ of $G$, we let $|A|$ be the order of $A$, that is, the dimension of the finite-dimensional $k$-vector space $\mathcal{O}(A)$; for a maximal torus $T$ of $G$, we let $A_T\coloneqq A\cap T$. When $\Char(k)=0$, the second author and Youssin showed that every finite abelian $p$-subgroup $A \subset G$ has a toral subgroup $A_T$ of index dividing $t(G)$; see \cite[Theorem 1.1]{reichstein-youssin2}. In particular, if $p$ is not a torsion prime, then $A = A_T$, i.e., $A$ itself is toral, and one recovers the above-mentioned result of Borel. In this paper we strengthen \cite[Theorem 1.1]{reichstein-youssin2} as follows.

\begin{theorem} \label{thm.t(G)} Let $k$ be an algebraically closed field, let $G$ be an affine $k$-group and let $A \subset G$ be a finite abelian $k$-subgroup such that $\Char(k)\nmid |A|$. Then there exists a maximal torus $T$ of $G$ such that $[A:A_T]$ divides the Grothendieck torsion index $t(G)$.
\end{theorem}

Theorem~\ref{thm.t(G)} answers a question of Totaro and Wang. This question was motivated by applications to the Mori Program; see~\cite{wang-jiahe}. The main novel aspect of Theorem~\ref{thm.t(G)} -- the one the application in~\cite{wang-jiahe} and most of the applications in this paper rely on -- is that $A$ is not required to be a $p$-group.  Note also that the proof of~\cite[Theorem 1.1]{reichstein-youssin2} uses resolution of singularities and thus is only valid in characteristic $0$. Here we allow  $\Char(k)$ to be positive, as long as it does not divide $|A|$. We will deduce Theorem~\ref{thm.t(G)} as a corollary to part (1) of Theorem~\ref{thm.main} below; see Section~\ref{sect.t(G)}.

\subsection{Torsors over fields of iterated Laurent series}
For any $r \geqslant 0$, we denote the field of iterated Laurent series in the variables $t_1, \ldots, t_r$ with coefficients in $k$ by  $k_r = k(\!(t_1)\!)\dots (\!(t_r)\!)$. 

\begin{theorem}\label{thm.main}
Let $k$ be an algebraically closed field, and let $G$ be an algebraic $k$-group. Let $A \subset G$ be a finite abelian $k$-subgroup such that $\Char(k) \not\mid |A|$, let $r\geqslant 0$ be an integer, let $E$ be a connected $A$-torsor over $k_r$, and let $E*_AG$ be the $G$-torsor over $k_r$ induced by $E$. Suppose a finite field extension $L/k_r$ splits $E*_AG$. 

 \begin{enumerate}
 \item There exists a maximal torus $T \subset G$ such that the index
 $[A: A_T]$ divides $[L : k_r]$.

 \item Assume further that there is a tower $k_r \subset  K \subset L$ of field extensions such that $[K : k_r]$ is prime to $|A|$ and $L/K$ is Galois. Then there exists a maximal torus $T \subset G$ such that (i) $A/A_T$ is isomorphic to a quotient of $\Gal(L/K)$ and (ii) $A/A_T$ is isomorphic to a subgroup of $\Gal(L/K)$. 
 \end{enumerate}
 \end{theorem}

Since $A$ is finite \'etale over $k$, the $A$-torsor $E$ is connected if and only if $E=\Spec(M)$, where $M$ is a Galois field extension of $k_r$ with Galois group $A(k)$.

In the case where $G = \PGL_{p^r}$, $A = \mu_p^{2r}$, part (2) of Theorem~\ref{thm.main} reduces to~\cite[Proposition 2.1]{genus1} and, if one further assumes that $K = k_r$, to \cite[Theorem 5]{amitsur91}. 

\subsection{Variants of the torsion index}
Let $G$ be an algebraic $k$-group. For any finite abelian $k$-subgroup $A\subset G$ we define the \emph{depth} of $A$ as the greatest common divisor of the indices $[A:A_T]$, where $T$ ranges over the maximal tori of $G$.
Suppose that $G$ is affine. In view of Theorem~\ref{thm.main}, it is natural to define the variants $t_2(G)$ and $t_3(G)$ of the Grothendieck torsion index as follows:
\begin{align}
t_2(G) &\coloneqq \mathrm{lcm}\{\mathrm{ind}(X)\mid r\in \mathbb{Z}_{>0},\ 
X \text{ a } G\text{-torsor over } k_r\}, \label{eq:t2}\\
t_3(G) &\coloneqq \mathrm{lcm}\{\mathrm{depth}(A)\mid
A\subset G \text{ finite abelian, $\Char(k)\nmid |A|$}\}. \label{eq:t3}
\end{align}
We have $t_3(G) \, | \, t_2(G)$ by Theorem~\ref{thm.main}(1) and
$t_2(G) \, | \, t(G)$ by the definition of $t(G)$.
We will show that, in fact, under a mild assumption on $\Char(k)$, $t_2(G)$ is always equal to $t_3(G)$; see Proposition~\ref{prop.torsion-a}. We will use this fact to prove the following.

\begin{theorem} \label{thm.torsion-b} Let $k$ be an algebraically closed field of characteristic zero, and let $G$ be an exceptional simple $k$-group of type $E_8$. Then $t_2(G) =  60$.
\end{theorem}

 Tits~\cite[Proposition 9]{tits1992} showed that for $G$ as in Theorem~\ref{thm.torsion-b}, $t(G)$ is divisible by $60$; for an alternative proof, see \cite[Corollary 5.5]{reichstein-youssin2}.
  Tits asked if, perhaps, $t(G) = 60$. He called this an ``hypoth\`ese optimiste''; see~\cite[Section 5.2]{tits92}. Totaro showed that $t(G)$ is, in fact, much larger, more specifically $t(G) = 2^6 3^2 5$; see~\cite{totaro-E8}. Theorem~\ref{thm.torsion-b} may be viewed as salvaging Tits' ``hypoth\`ese optimiste'' with $t(G)$ replaced by $t_2(G)$. In particular, it shows that $t_2(G)$ may be strictly smaller than $t(G)$.

We also note that Tits~\cite{tits92} and Totaro~\cite{totaro-E8, totaro-spin} computed
$t(G)$ for the simply connected simple groups $G$ over the algebraically closed field $k$. In a similar spirit, it would be interesting to compute $t_2(G)$ for all simple groups. Theorem~\ref{thm.torsion-b} is a step in this direction.
The case of $G = \Spin_n$ looks particularly intriguing, because of the connection between finite abelian subgroups of $\Spin_n$ and self-dual binary error-correcting codes of length $n$; see~\cite{wood}.

\subsection{Zero-cycles vs. rational points}
The following natural question was raised by Totaro in 
\cite[Question 0.2]{totaro-splitting}.

\begin{question}\label{Totaro_question}
    Let $K$ be a field, let $G$ be a smooth affine connected $K$-group, and let $X$ be a $G$-homogeneous space over $K$. If $X$ has a zero-cycle of degree $d$ over $K$, does it have a point of degree dividing $d$? In other words, if
    $X(L_i) \neq \emptyset$ for finite field extensions $L_1,\dots,L_n$ of $K$ and $\gcd([L_1:K],\dots,[L_n:K]) = d$, does it follow that there exists a field extension $L/K$ whose degree $[L:K]$ divides $d$ such that $X(L) \neq \emptyset$? 
\end{question}

In the case where $d = 1$ and $X$ is projective, Question~\ref{Totaro_question} was asked earlier by Veisfeiler~\cite{veisfeiler}. In the case where $d=1$ and $X$ is a principal homogeneous space (i.e., a $G$-torsor) over $K$, Question~\ref{Totaro_question} goes back to Serre~\cite[Section 2.4, Question 2]{serre-pp}.

Counterexamples to Question~\ref{Totaro_question} (with $d = 1$) were first constructed by Florence~\cite{florence-zero-cycles}. One of these counterexamples is over the field of iterated Laurent series $\mathbb{C}_2 = \mathbb{C}(\!(t_1)\!)(\!(t_2)\!)$. Our next result gives a positive answer to Question~\ref{Totaro_question} in the case where $X$ is a principal homogeneous space (i.e., a torsor) over the field  $k_r = k(\!(t_1)\!)\dots(\!(t_r)\!)$ of iterated Laurent series.

\begin{theorem}\label{thm.serre_quest}
    Let $k$ be an algebraically closed field, and let $G$ be a smooth affine $k$-group.
    Assume that the characteristic of $k$ is good for $G$; see Definition~\ref{def.good_char}. Then, for all $d,r\geqslant 1$, every $G$-torsor $X$ over $k_r$ with a zero-cycle of degree $d$ has a closed point whose degree divides $d$.
\end{theorem}

When $G$ is not connected, there are counterexamples to Question~\ref{Totaro_question}  even in the case where $X$ is a principal homogeneous space and $d = 1$; see \cite[p.~192]{serre-gc}. Theorem~\ref{thm.serre_quest} shows that these counterexamples do not occur over $K=k_r$.

\subsection{Splitting torsors by genus 1 curves}
	By a genus $1$ curve over a field $K$, we shall mean a smooth projective geometrically connected curve of geometric genus $1$ over $K$. We shall say that a Brauer class $\alpha\in \on{Br}(K)$ is split by a smooth integral $K$-variety $Y$ if $\alpha$ pulls back to zero in $\on{Br}(Y)$. By a theorem of Grothendieck, this is equivalent to $\alpha$ pulling back to zero in $\on{Br}(K(Y))$. The following question was asked independently by Clark \cite{clark2008open} and Saltman \cite{ruozzi2011rage}.
	
	\begin{question} \label{question-clark-saltman}
		Let $K$ be a field, and let $\alpha\in \Br(K)$ be a Brauer class. Does there exist a genus $1$ curve $C$ over $K$ such that $\alpha$ is split by $C$?
	\end{question}

This question was recently answered in the negative by the second and third authors~\cite{genus1}. It is natural to generalize Question~\ref{question-clark-saltman} in two ways. 

\smallskip
(i) Any genus $1$ curve $C$ is a torsor over its Jacobian. Fixing an integer $g \geqslant 1$, we can ask a similar question in the case where $C$ is allowed to be a torsor over a $g$-dimensional abelian variety. 

\smallskip
(ii) Recall that
a Brauer class $\alpha \in \Br(K)$ is uniquely represented by a $\PGL_n$-torsor
$X$ over $K$, where $n$ is the Schur index of $\alpha$. Moreover, the genus $1$ curve $C$ splits $\alpha$ if and only if it splits $X$. It is thus natural to replace the Brauer class $\alpha$ by a torsor under some smooth affine connected $k$-group $G$, not necessarily $\PGL_n$. 

\smallskip
\noindent Note that (i) was already considered in~\cite{genus1}, whereas (ii) was not. Combining (i) and (ii), we arrive at the following more general question.

	\begin{question} \label{question-clark-saltman2}
    Let $K$ be a field, let $G$ be a smooth affine connected $K$-group, and let
    $X$ be a $G$-torsor over $K$. Fix an integer $g \geqslant 1$.  Does there exist a $g$-dimensional abelian variety $\Lambda$ over $K$ and a $\Lambda$-torsor $C$ over $K$ such that $C$ splits $X$?
    \end{question}

 The following result, generalizing \cite[Theorem 1.2]{genus1}, shows that the answer to Question~\ref{question-clark-saltman2} is negative when $K$ is a field extension of an algebraically closed field $k$ and $G$ is both defined over $k$ and contains an elementary abelian $p$-group of sufficiently large depth.
 
\begin{theorem}\label{thm.non-toral-obstruction}
Let $k$ be an algebraically closed field, let $p$ be a prime not equal to
$\operatorname{char}(k)$, let $G$ be an algebraic $k$-group, let $A$ be an elementary abelian $p$-subgroup of $G$, let $r$ be the rank of $A$, and let $E$ be a connected $A$-torsor over $k_r$. Let $g\geqslant 1$ be an integer, and suppose that for every maximal torus $T\subset G$,
\[
        \mathrm{rank}(A/A_T) \geqslant 
        \begin{cases}
        5g^2+2g, & \text{if } p>2 \text{ and } g>1,\\[2mm]
        9g^2+2g-1, & \text{if } p=2 \text{ and } g>1,\\[2mm]
        6, & \text{if } p>2 \text{ and } g=1,\\[2mm]
        7, & \text{if } p=2 \text{ and } g=1.
        \end{cases}
\]
Then the induced $G$-torsor $E*_AG$ over $k_r$ cannot be split by any torsor under any $g$-dimensional abelian variety over $k_r$.
\end{theorem}

Our proof of Theorem~\ref{thm.non-toral-obstruction} is based on Theorem~\ref{thm.main}(2), which generalizes \cite[Proposition 2.1]{genus1}. We give some applications of Theorem~\ref{thm.non-toral-obstruction} in Section~\ref{sect.genus1}.

\section{Notation and Preliminaries}
\label{sect.prelim} 

\subsection{Notational conventions}
\label{subsect.notation}
Throughout this article, we fix an algebraically closed field $k$. For every integer $r\geqslant 0$, we let $k_r \coloneqq k(\!(t_1)\!)\dots (\!(t_r)\!)$ be the field of iterated Laurent series in the variables $t_1, \ldots, t_r$, with the convention $k_0=k$.

Let $G$ be an algebraic $k$-group, that is, a group scheme of finite type over $k$. By definition, the \emph{smooth linear part} $\Glin\subset G$ of $G$ is the largest smooth affine and connected $k$-subgroup of $G$; see \cite[Section 2.1]{bouthier2025generically}. Moreover, $\Glin$ is normal in the maximal reduced subgroup $G_{\mathrm{red}} \subset G$ and the quotient $G_{\mathrm{red}}^{\circ}/\Glin$ is an abelian variety. If $H\subset G$ is a finite-index closed $k$-subgroup of $G$, we let $[G:H]$ denote the index of $H$ in $G$. Recall that by definition, the algebraic $k$-group $G$ is reductive if it is smooth, affine, connected and its unipotent radical $R_u(G)$ is trivial.

Let $A$ be a finite abelian $k$-group.  We let $|A|$ be the order of $A$, that is, the dimension of the finite-dimensional $k$-vector space $\mathcal{O}(A)$. If $A$ is \'etale (for example, if $|A|$ is invertible in $k$), then $|A|=|A(k)|$. If $A\subset G$, for every $k$-torus $T\subset G$ we let $A_T\coloneqq A\cap T$, and we say that $A$ is \emph{toral} if $A$ is contained in a $k$-torus of $G$.

 \subsection{The Grothendieck torsion index}
The Grothendieck torsion index $t(G)$ of an affine algebraic $k$-group $G$ was defined in \eqref{eq:t}.

\begin{proposition} \label{prop:tG-well-defined}
   Let $G$ be an affine $k$-group, let $E \to Z$ be a versal $G$-torsor, where $Z$ is an integral $k$-scheme of finite type. Then
\[t(G) = \operatorname{ind}(E_{k(Z)}).\]
In particular, $t(G)$ is an integer.
\end{proposition}

For the definition of a versal torsor, see \cite[Definition 5.1]{serre-ci}. In the case where $G$ is assumed to be connected, Proposition~\ref{prop:tG-well-defined} follows from~\cite[Theorem 2]{grothendieck-special}. See also \cite[Theorem 1.1]{totaro-spin}, where $G$ is assumed to be reductive. Here $G$ is an arbitrary affine $k$-group.

\begin{proof}
It is clear that $\operatorname{ind}(E_{k(Z)}) \mid t(G)$. Conversely, let $K/k$ be a field extension, and let $X$ be a $G$-torsor over $K$. Let $L/k(Z)$ be a finite field extension which splits $E_{k(Z)}$, and let $d\coloneqq [L:k(Z)]$. By generic flatness, after replacing $Z$ by a non-empty open subset, the extension $L/k(Z)$ spreads out to a finite flat morphism $Z'\to Z$ of degree $d$, such that the pullback of $E$ to $Z'$ is trivial.

By the versality of $E\to Z$, there exists $b\in Z(K)$ such that $b^*E\cong X$. Consider the finite $K$-algebra $A \coloneqq K \times_Z Z'$. By functoriality, $X_A$ is a trivial $G$-torsor over $A$. Since $Z' \to Z$ is finite flat of degree $d$, we have $\dim_K(A)=d$.  Let $\kappa_1,\dots,\kappa_r$ be the residue fields of the Artinian $K$-algebra $A$, and let $m_i$ be the corresponding lengths. Then
\[
d = \dim_K A = \sum_{i=1}^r m_i [\kappa_i:K].
\]
Since $X_A$ is trivial, $X_{\kappa_i}$ is trivial for every $i$. Therefore $\operatorname{ind}(X) \mid [\kappa_i:K]$ for every $i$. Hence $\operatorname{ind}(X)$ divides $\sum_{i=1}^r m_i[\kappa_i:K] = d=[L:k(Z)]$. Since this holds for every splitting field $L$ of $E_{k(Z)}$, we deduce that $\ind(X)\mid \ind(E_{k(Z)})$. Thus $t(G)\mid \ind(E_{k(Z)})$ and hence $t(G)=\ind(E_{k(Z)})$, as desired.
\end{proof}

\begin{lemma} \label{lem.t(G)}
Let 
\[1 \longrightarrow G_1 \longrightarrow G \longrightarrow G_2 \longrightarrow 1\] be a short exact sequence of affine algebraic $k$-groups. Then $t(G)$ divides $t(G_1) t(G_2)$.
\end{lemma}

\begin{proof} Let $K$ be a field containing $k$, and let $E$ be a $G$-torsor over $K$.  
Letting $E_2=E/G_1$ be the $G_2$-torsor over $K$ induced by $E$, by definition of $t(G_2)$ there exist field extensions $K_1/K, \ldots, K_m/K$ of degrees $d_i\coloneqq [K_i : K]$ such that $E_2(K_i)\neq\emptyset$ for all $i=1,\dots,m$ and 
\begin{equation} \label{e.t(G_2)}
\text{$\gcd(d_1, \ldots, d_m)$ divides $t(G_2)$.}
\end{equation}
For all $i=1,\dots,m$, since $E_2(K_i)\neq\emptyset$, the $G$-torsor $E_{K_i}$ admits reduction of structure to a $G_1$-torsor $\tilde{E}_i$ over $K_i$. By definition of $t(G_1)$, for all $i = 1, \ldots, m$ there exist field extensions $K_{i\, 1}/K_i, \ldots, K_{i\,n_i}/K_i$ of degrees $d_{ij} \coloneqq [K_{ij}: K_i]$ such that each $\tilde{E}_i(K_{ij})\neq\emptyset$, and
\begin{equation} \label{e.t(G_1)}
\text{$\gcd(d_{i\, 1}, \ldots, d_{i\, n_i})$ divides $t(G_1)$.}
\end{equation} 
For all $i,j$, since $\tilde{E}_i(K_{ij})\neq\emptyset$, we have that $E(K_{ij})\neq\emptyset$, and hence $\ind(E)$ divides  
\[ [K_{ij}: K] = [K_{ij}: K_i] \cdot [K_i: K] = d_{ij} \cdot d_i.\]
Therefore, in order to finish the proof of Lemma~\ref{lem.t(G)}, it suffices to show that the ideal $\Lambda$ of $\mathbb Z$ generated by $d_i d_{ij}$ for every $i = 1, \ldots, m$ and $j = 1, \ldots, n_i$ contains $t(G_1) t(G_2)$. To prove this, let $I \coloneqq (d_1, \ldots, d_m) $ be the ideal of $\mathbb Z$ generated by $d_1, \ldots, d_m$ and $J_i \coloneqq (d_{i\, 1}, \ldots, d_{i\, n_i})$ be the ideal of $\mathbb Z$ generated by $d_{i\, 1}, \ldots, d_{i\, n_i}$. Then $(t(G_2)) \subset I$ by~\eqref{e.t(G_2)}, and $(t(G_1)) \subset J_i$ for each $i$  by~\eqref{e.t(G_1)}. Now
\begin{align*}
(t(G_1)t(G_2))
&= (t(G_1)) \cdot (t(G_2)) \\
&\subset (t(G_1)) \cdot I \\
&= (t(G_1)) \cdot (d_1) + (t(G_1)) \cdot (d_2) + \ldots
   + (t(G_1)) \cdot (d_m) \\
&\subset J_1 \cdot (d_1) + J_2 \cdot (d_2) + \ldots + J_m \cdot (d_m) \\
&= \Lambda,
\end{align*}
as desired.     
\end{proof}

\begin{remark} The same argument shows that $t_2(G)$ divides $t_2(G_1)t_2(G_2)$. The only difference is that instead of taking $K$ to be an arbitrary field containing $k$, we take $K$ to be $k_r$ for some $r \geqslant 0$. 
\end{remark}

\begin{remark} \label{rem.chow} 
    Let $G$ be a reductive $k$-group, let $B$ be a Borel subgroup of $G$, and let $T \subset B$ be a maximal torus. In \cite[Chapter 4]{grothendieck-special}, Grothendieck defined the torsion index $t(G)$ in terms of the characteristic map 
\[ c_G \colon \operatorname{Sym}(X(T)) \longrightarrow \operatorname{CH}(G/B)\]
and showed that this definition is equivalent to \eqref{eq:t}; see \cite[Theorem 2]{grothendieck-special}.
\end{remark}

\subsection{Torsion primes}
Let $G$ be a reductive $k$-group. (Recall that this implies that $G$ is connected.) There are several definitions of a torsion prime for $G$ in the literature.

\smallskip
(i) Grothendieck's definition~\cite[Definition 3]{grothendieck-special}. A prime integer $p$ is called a torsion prime for $G$ if $p$ divides the Grothendieck torsion index $t(G)$.

\smallskip
(ii) Steinberg's definition~\cite[Definition 2.1]{steinberg-torsion}. Let $H$ be a closed subgroup of $G$. Assume that $H$ is regular in $G$, i.e., $H$ contains a maximal torus of $G$. Then we can identify the root system of $H$ as a subsystem of the root system of $G$. We will say that a regular subgroup $H$ is admissible if its root system is integrally closed in the root system of $G$. A prime $p$ is called a torsion prime for $G$ if it divides the order of the fundamental group $F(H')$ of $H'$ for some regular admissible subgroup $H$ of $G$. Here $H'$ denotes the derived subgroup of $H$.

\smallskip
(iii) The root system definition. Let $\Sigma$ be the root system of $G$, $\Sigma^* = \{ \alpha^* = \frac{2 \alpha}{(\alpha, \alpha)} \, | \, \alpha \in \Sigma \}$ be the dual root system and $L(\Sigma^*)$ be the lattice generated by $\Sigma^*$. 
A prime $p$ is called a torsion prime for $G$ if $L(\Sigma^*)/L(\Sigma_1^*)$ has non-trivial $p$-torsion for some closed subsystem $\Sigma_1 \subset \Sigma$ or if $p$ divides the order of the fundamental group $F(G')$, where $G'$ is the derived subgroup of $G$.

\smallskip
(iv) Borel's definition~\cite{borel-tohoku}. Here we assume that $k =\mathbb{C}$. Let $K$ be a maximal compact subgroup of $G$. Then $p$ is called a torsion prime for $G$ if $H^*(K, \mathbb Z)$ has non-trivial additive $p$-torsion.

\begin{lemma} \label{lem.torsion-primes}
Let $G$ be a split reductive group scheme over $\bZ$. Specializing $G$ to the algebraically closed field $k$, we obtain the reductive $k$-group $G_k = G \times_{\bZ} k$.

\begin{enumerate}
    \item The sets of torsion primes for $G_k$ in the sense of definitions {\rm (i)}-{\rm (iii)}  coincide. If $k=\mathbb C$, then this common set also coincides with the set of torsion primes obtained in the sense of definition {\rm (iv)}.
    \item Let $k'$ be an algebraically closed field. Then a prime $p$ is a torsion prime for $G_k$ if and only if it is a torsion prime for $G_{k'}$.  
\end{enumerate}
\end{lemma}

\begin{proof}
Definition (iii) is purely in terms of the root system of $G$, hence, is independent of the choice of $k$.

Steinberg~\cite[Lemma 2.5]{steinberg-torsion}  showed that (ii) is equivalent to (iii); hence, also, independent of the choice of $k$.  Grothendieck showed that (i) is also independent of the choice of $k$ (see the proof of \cite[Theorem 3]{grothendieck-special}) and that for $k = \mathbb C$, (i) is equivalent to (iv); see \cite[Theorem 4]{grothendieck-special}. 

It remains to show that (ii) is equivalent to (iv) when $k = \mathbb C$. By \cite[Theorem 2.28]{steinberg-torsion}, 
$p$ is a torsion prime for $G$ in the sense of (ii) if and only if 
\begin{equation} \label{e.non-toral-algebraic}
\text{$G$ has a non-toral finite elementary abelian $p$-subgroup.}
\end{equation}
On the other hand, by \cite[Theorem 2.28]{borel-tohoku}, $p$ is 
a torsion prime for $G$ in the sense of (iv) if and only if 
\begin{equation} \label{e.non-toral-compact}
\text{$K$ has a non-toral finite elementary abelian $p$-subgroup.}
\end{equation}
Conditions~\eqref{e.non-toral-algebraic} and~\eqref{e.non-toral-compact} are readily seen to be equivalent.
\end{proof}

\begin{remark} Lemma~\ref{lem.torsion-primes} will not be cited in this paper. However, it will be used implicitly to justify citing various sources, where a priori different definitions of torsion primes may be used.
\end{remark}

\begin{remark} Recall that the base field $k$ is assumed to be algebraically closed throughout this paper. Some authors consider the Grothendieck torsion index and torsion primes for algebraic groups defined over an arbitrary field;
see, e.g., \cite{tits1992, serre-pp}. We will not do this here.
\end{remark}

\subsection{Galois theory of iterated Laurent series fields}
\label{sect.galois-laurent}
Let $r\geqslant 0$ be an integer, and let $l$ be the characteristic exponent of $k$, that is, $l = 1$ if $\Char(k) = 0$ and $l = \Char(k)$ if $\Char(k) > 0$.
For any $r$-tuple of prime-to-$l$ positive integers $n = (n_1,\dots,n_r)$, let $F_n = k(\!(s_1)\!)\dots(\!(s_r)\!)$ be the field extension of $k_r$ given by $s_i^{n_i} = t_i$ for all $1\leq i\leq r$. Then $F_n/k_r$ is a tamely ramified Galois extension of $k_r$ with Galois group $$\Gal(F_n/k_r)\cong \mu_{n_1}\times \dots \times \mu_{n_r}.$$
Let $\hat{\bZ}'(1)\coloneqq \varprojlim \mu_m$, where $m$ ranges over all integers prime to $l$. Taking the limit of the surjections $\Gal(k_r) \twoheadrightarrow \Gal(F_n/k_r)$ over all possible $r$-tuples $n$ induces a short exact sequence
\begin{equation}\label{eq:gal-kr-presentation}
    1 \longrightarrow P \longrightarrow \Gal(k_r) \xlongrightarrow{\pi} \hat{\bZ}'(1)^r \longrightarrow 1,
\end{equation}
where $P$ is a pro-$l$ group; see \cite[Theorem A.24]{tignol2015value} or \cite[Lemma 5.1]{gille2009lower}. (In particular, $\pi$ is an isomorphism if $\Char(k)=0$.) In other words, $\pi$ identifies the maximal prime-to-$l$ quotient of $\Gal(k_r)$ with $\hat{\bZ}'(1)^r$.

Let $G$ be an algebraic $k$-group, and let $\vphi\colon \Gal(k_r)\to G(k)$ be a continuous homomorphism. The homomorphism $\vphi$ is in particular a cocycle for $\Gal(k_r)$, and hence it defines a cohomology class $[\vphi]\in H^1(k_r,G)$. We say that $\vphi\colon \Gal(k_r)\to G(k)$ is \emph{tame} if the order of $\mathrm{Im}(\vphi)\subset G(k)$ is prime to $l$.

\begin{lemma}\label{lem.tame_hom_is_loop}
Let $G$ be an algebraic $k$-group, let $r\geqslant 0$, and let $\vphi\colon \Gal(k_r)\to G(k)$ be a tame continuous homomorphism. Let $A_\vphi\subset G$ be the unique finite \'etale abelian $k$-subgroup of $G$ such that $A_\vphi(k)=\mathrm{Im}(\vphi)$, so that $\Char(k)\nmid |A_{\vphi}|$. Then there exists a connected $A_\vphi$-torsor $E$ over $k_r$ such that $[E*_{A_\vphi}G]=[\vphi]$ in $H^1(k_r,G)$.
\end{lemma}
\begin{proof}
Since $\vphi$ is tame, its image has order prime to $l$. Hence the
pro-$l$ subgroup $P$ in \eqref{eq:gal-kr-presentation} is contained in
$\ker(\vphi)$, and $\vphi$ factors through the maximal prime-to-$l$
quotient of $\Gal(k_r)$. Let $M/k_r$ be the finite Galois extension corresponding to $\ker(\vphi)$. Then $\vphi$ induces an isomorphism
\[
    \bar{\vphi}\colon \Gal(M/k_r) \xlongrightarrow{\sim} A_\vphi(k).
\]
Let $E$ be the $A_\vphi(k)$-torsor over $k_r$ induced by the $\Gal(M/k_r)$-torsor $\Spec(M)\to\Spec(k_r)$ via $\bar{\vphi}$. Since $M$ is a field and $\bar{\vphi}$ is an isomorphism, $E\coloneqq\Spec(M)$ is connected. By definition, the induced $G$-torsor
$E*_{A_\vphi}G$ is represented by the composite cocycle
\[
  \Gal(k_r)
  \mathrel{\relbar\joinrel\twoheadrightarrow}
  \Gal(M/k_r)
  \xlongrightarrow{\bar{\vphi}}
  A_\vphi(k)
  \mathrel{\lhook\joinrel\longrightarrow}
  G(k)
\]
which by definition is equal to $\vphi$. Therefore
$[E*_{A_\vphi}G]=[\vphi]$ in $H^1(k_r,G)$, as desired.
\end{proof}
Under mild assumptions on $G$, every $G$-torsor over $k_r$ admits reduction of structure to a connected torsor under a finite \'etale abelian $k$-subgroup of $G$.

\begin{definition}\label{def.good_char}
    Let $G$ be a smooth affine algebraic $k$-group, let $R_u(G) \subset G$ be its unipotent radical and set $\ol{G} = G/R_u(G)$. We say that \emph{$\Char(k)$ is good for $G$} if $\Char(k)$ does not divide $|W(\ol{G})| \cdot [\ol{G}:\ol{G}^{\,\circ}]$, where
    $W(\ol{G})$ and $\ol{G}^{\, \circ}$ denote the Weyl group and the connected component of $\ol{G}$, respectively.
\end{definition}

\begin{lemma}\label{lem.serre_quest}
Let $G$ be a smooth affine algebraic $k$-group, let $r\geqslant 0$ be an integer, and let $X$ be a $G$-torsor over $k_r$.   Assume that $\Char(k)$ is good for $G$ as in Definition~\ref{def.good_char}. Then there exist a finite abelian $k$-subgroup $A \subset G$ of order prime to $\Char(k)$ and a connected $A$-torsor $E$ over $k_r$ such that $X\cong E*_AG$. In particular, $X$ admits reduction of structure to $A$.
\end{lemma}

\begin{proof}
Set $\ol{G} \coloneqq G/R_u(G)$, where $R_u(G)$ is the unipotent radical of $G$. We break the proof into two steps.

\smallskip
(1) We start with the special case where $G^{\circ}$ is reductive. Equivalently, we assume $\ol{G} = G$. By \cite[Corollaire 18]{lucchini2015groupe}, there exists a finite $k$-subgroup $S\subset G$ of order prime to $\Char(k)$ such that for every field extension $K/k$, the map $H^1(K,S)\to H^1(K,G)$ is surjective. In particular, $X$ admits reduction of structure to $S$, that is, $[X] = [\vphi]$ in $H^1(k_r,G)$ for some continuous homomorphism $\vphi\colon \Gal(k_r)\to G(k)$ taking values in $S$. Since $\Char(k)\nmid |S|$, the homomorphism $\vphi$ is tame. The conclusion now follows from Lemma~\ref{lem.tame_hom_is_loop}.

\smallskip
(2) Assume $G^{\circ}$ is not reductive and let $\pi\colon G\to \ol{G}$ be the canonical surjection. By the previous step, there exist a finite abelian $k$-subgroup $A\subset G$ of order prime to  $\Char(k)$ and a connected $A$-torsor $E$ such that $\pi_*(X) \cong E \ast_A \ol{G}$. Write $j\colon A\to \ol{G}$ for the inclusion map. The extension
$$ 1\longrightarrow R_u(G) \longrightarrow \pi^{-1}(A) \longrightarrow A \longrightarrow 1$$
splits by \cite[Expos\'e XVII, Théorème 5.1.1]{SGA3}. Thus, we obtain an embedding $i\colon A \to G$ such that $\pi\circ i = j$ is the given inclusion of $A$ into $\ol{G}$. Therefore
\begin{equation}\label{e.lem_serre_quest}
\pi_*[E\ast_A G] = [E\ast_{A} \ol{G}] = \pi_*[X],
\end{equation}
where $\pi_*\colon H^1(k_r,G)\to H^1(k_r,\ol{G})$ is the pushforward map.
Assume $[a] = [E]$ for some $\Gal(k_r)$-cocycle $a$ with values in $A(k)$. Let ${}_{a}R_{u}(G)$ be the twisted unipotent group defined by $a$ \cite[Chapter I, Section 5.3]{serre1997galois}. Then ${}_{a}R_{u}(G)$ becomes split over a separable closure of $k_r$, and hence must already be split over $k_r$ by \cite[Theorem B.3.4]{conrad-gabber-prasad}. In particular, we have $H^1(k_r, {}_{a} R_u(G)) = \{\ast\}$. Now \eqref{e.lem_serre_quest} implies $X \cong E\ast_A G$ by a standard twisting argument; see \cite[Chapter I, Section 5.5, Corollary 2]{serre1997galois}.
\end{proof}

\section{Proof of Theorem~\ref{thm.main}} %

We first prove Theorem~\ref{thm.main} in the special case, where
$L = k_r$. Under this assumption, both parts reduce to the following. 

\begin{proposition}\label{prop.special_case}
    Let $G$ be an algebraic $k$-group and let $A \subset G$ be a finite abelian $k$-subgroup of rank $r$.
    Assume that $\Char(k)\not\mid |A|$.
    Then the following conditions are equivalent.
    
    \begin{enumerate}
        \item[(a)] The abelian $k$-subgroup $A$ is toral.
        \item[(b)] For every field extension $K/k$, the map of pointed sets $H^1(K, A) \to H^1(K, G)$ is trivial.
        \item[(c)] The map of pointed sets $H^1(k_r, A) \to H^1(k_r, G)$ is trivial.
        \item[(d)] For some (equivalently, every) connected $A$-torsor $E$ over $k_r$, the induced $G$-torsor $E*_AG$ is split.
    \end{enumerate}
\end{proposition}

In (d), note that a connected $A$-torsor over $k_r$ exists: indeed, this is equivalent to the existence of a surjection $\mathrm{Gal}(k_r)\to A(k)$, which follows from \eqref{eq:gal-kr-presentation}.

In the case where $G^{\circ}$ is reductive but $k$ is an arbitrary field, Proposition~\ref{prop.special_case} was proved by Gille using Bruhat-Tits theory; see \cite[Corollary 4.15]{gille:hal-04621048}.

\begin{proof}
(a) $\Longrightarrow$ (b). Suppose $A$ lies in a torus $T \subset G$. Then the morphism $H^1(K, A) \to H^1(K, G)$ factors through $H^1(K, T)$,
and $H^1(K, T) = 1$ by Hilbert's Theorem 90.

\smallskip
(b) $\Longrightarrow$ (c). This is immediate.
\smallskip

(c) $\Longrightarrow$ (d). By definition $[E*_AG]$ is the image of 
$[E]$ under the map $H^1(k_r, A) \to H^1(k_r, G)$, and hence by (c) the $G$-torsor $E*_AG$ is split.

\smallskip
(d) $\Longrightarrow$ (a). Let $H = \Glin \subset G$ be the smooth linear part of $G$ (see \ref{subsect.notation}).
 Note that the morphism $G/H \to G/G_{\mathrm{red}}^{\circ}$ is proper. Indeed, after the fppf base change $G\to G/G^{\circ}_{\mathrm{red}}$, it becomes the first projection $G\times_k (G^{\circ}_{\mathrm{red}}/H)\to G$, which is proper because  $G^{\circ}_{\mathrm{red}}/H$ is an abelian variety. Choose a Borel subgroup $B \subset H$. A similar descent argument shows $G/B \to G/H$  is proper because $H/B$ is proper. Since $G/G^{\circ}_{\mathrm{red}} \to \Spec k$ is finite, the composition
 $$G/B \to G/H \to G/G^{\circ}_{\mathrm{red}} \to \Spec k$$
 is also proper. We conclude that $G/B$ is proper.
 
The inclusion $A\hookrightarrow G$ induces an $A$-equivariant map $E=E*_AA\hookrightarrow E*_AG$. Since $E*_AG$ is split over $k_r$, there exists a $G$-equivariant isomorphism $E*_AG\xrightarrow{\sim} G_{k_r}$ over $k_r$. Composing this isomorphism with the natural projection $G_{k_r} \to (G/B)_{k_r}$, we obtain an $A$-equivariant morphism $\pi \colon E \to G/B$. 

 Since $\Char(k)\nmid |A|$, the connected $A$-torsor $E$ is \'etale, and hence it is of the form $E = \Spec(K)$ for some Galois field extension $K/k_r$ with Galois group $A(k)$. Since the rank-$r$ valuation on $k_r$ is henselian, it extends uniquely to a rank-$r$ valuation $v$ on $K$. The uniqueness of $v$ implies that $v$ is $A(k)$-invariant. Let $V\subset K$ be the valuation ring of $v$. Then $V$ is a local ring with fraction field $K$ and residue field $k$. Since $v$ is $A(k)$-invariant, $V$ is $A(k)$-stable, and hence the canonical morphism $E=\Spec(K)\longrightarrow \Spec(V)$ is $A(k)$-equivariant, or equivalently, $A$-equivariant. In summary, we obtain the following commutative square of $A$-equivariant maps, where $A$ acts trivially on $\Spec(k)$:
\[
\begin{tikzcd}
E \arrow[r, "\pi"] \arrow[d] & G/B \arrow[d] \\
\Spec(V) \arrow[r] & \Spec(k).
\end{tikzcd}
\]
Since $G/B$ is proper over $k$, the valuative criterion for properness tells us that there is a morphism
$h \colon \Spec(V) \to G/B$ such that the resulting diagram
\[
\begin{tikzcd}
E \arrow[r, "\pi"] \arrow[d] & G/B \arrow[d] \\
\Spec(V) \arrow[ur, "h"] \arrow[r] & \Spec(k)
\end{tikzcd}
\]
commutes. Since $\pi$ is $A$-equivariant, $h$ is $A$-equivariant as well. In particular, since the closed point of $V$ corresponding to the maximal ideal of $v$ is fixed by $A$, so is its image in $G/B$. Therefore, there exists $g\in G(k)$ such that $gB \in (G/B)(k)$ is fixed by $A$. We conclude that $A$ is contained in the $G$-stabilizer $B'\coloneqq gBg^{-1}$ of $gB$. Since $B$ is a Borel subgroup of $H = \Glin$, so is $B'$. Since $\Char(k)\nmid |A|$ and $k$ is algebraically closed, the abelian $k$-subgroup $A$ lies in a maximal torus of $B'$, and hence of $G$, as desired.
\end{proof}

\begin{proof}[Conclusion of the proof of Theorem~\ref{thm.main}]
Since $E$ is connected, $E = \Spec(K)$ for some Galois field extension $K/k_r$ with Galois group $A(k)$. The extension $K/k_r$ defines a surjective continuous homomorphism $\vphi'\colon\Gal(k_r)\to A(k)$. Let $\vphi\colon\Gal(k_r) \to G(k)$ be the composite of $\vphi'$ with the inclusion $A(k) \subset G(k)$. By construction, we have $[E*_AG] = [\vphi]$ in $H^1(k_r,G)$.

Let $L/k_r$ be a finite splitting field of $E*_AG$, consider $\Gal(L)$ as a subgroup of $\Gal(k_r)$ via the restriction map, and let $\vphi_L$ be the restriction of $\vphi$ to $\Gal(L)$, so that $[(E*_AG)_L] = [\vphi_{L}]$ by functoriality. The fields $L$ and $k_r$ are isomorphic over $k$; see \cite[Corollary 5.4]{gille2009lower}. Moreover, by Lemma~\ref{lem.tame_hom_is_loop} the class $[\vphi_{L}]\in H^1(L,G)$ is induced from a connected $A_{\vphi_L}$-torsor $E'$ over $L$.  Therefore $[E'\ast_{A_{\vphi_L}}G] = [\vphi_{L}]$ being split implies that $A_{\vphi_L}$ is contained in a torus $T \subset G$ by Proposition~\ref{prop.special_case}. Consider the composite surjection
$$\overline{\vphi}\colon \Gal(k_r) \xlongrightarrow{\vphi'} A(k) \longrightarrow A(k)/A_T(k).$$
Since $\Gal(L) \subset \ker \overline{\vphi}$, the index $[A:A_T]$ divides $[\Gal(k_r):\Gal(L)]$. Now (1) of the theorem follows because $[\Gal(k_r):\Gal(L)]$ is the separable degree of the extension $L/k_r$ and hence divides $[L:k_r]$.

\smallskip
Assume further that the assumptions of (2) are satisfied. Note that $[\Gal(k_r):\Gal(K)]$ divides $[K:k_r]$ because it is the separable degree of the extension $K/k_r$. Therefore, under the assumptions of (2), $[\Gal(k_r):\Gal(K)]$ is prime to $|A|$. We conclude that the restriction of $\ol{\vphi}$ to $\Gal(K)$ is surjective. Let $A'\coloneqq A(k)/A_T(k)$. Since $\vphi(\Gal(L))\subset A_T(k)$, quotienting by $\Gal(L)$ gives a surjective homomorphism of Galois groups:
\begin{equation}\label{e.A/A_T_is_quotient}
    \Gal(L/K) = \Gal(K)/\Gal(L) \twoheadrightarrow  A'.
\end{equation}
Therefore $A'$ is a homomorphic image of $\Gal(L/K)$. It remains to show that $A'$ is isomorphic to a subgroup of $\Gal(L/K)$. Note that $K$ and $k_r$ are isomorphic as fields over $k$ by \cite[Corollary 6.4]{gille2009lower}.
It follows from~\eqref{eq:gal-kr-presentation}
that $\Gal(L/K)$ is a semi-direct product $\Gal(L/K) = \mathcal{P} \rtimes \mathcal{A}$ for some finite abelian subgroup $\mathcal{A}$ of order prime to $l$ and a normal subgroup $\mathcal{P}$ of order a power of $l$. Here $l$ denotes the characteristic exponent of $k$, as in
Section~\ref{sect.galois-laurent}. 


Since $A'$ is of order prime to $l$, the surjection \eqref{e.A/A_T_is_quotient} restricts to a surjection $\mathcal{A} \twoheadrightarrow A'$. Since $\mathcal{A}$ is a finite abelian group, the existence of a surjection $\mathcal{A} \twoheadrightarrow A'$ implies the existence of an embedding $A'\hookrightarrow \mathcal{A}$ by the structure theorem of finite abelian groups. Composing with the inclusion into $\Gal(L/K)$, we see that $A'$ is isomorphic to a subgroup of $\Gal(L/K)$.
\end{proof}

\section{Variations on the torsion index}
\label{sect.t_2=t_3}

In \eqref{eq:t}-\eqref{eq:t3}, we defined the torsion indices $t(G)$, $t_2(G)$ and $t_3(G)$ for any affine $k$-group $G$. As we pointed out in the Introduction, we have 
\begin{equation} \label{e.divisibilities}
t_3(G) \mid t_2(G) \mid t(G).
\end{equation}
In this section we show that under mild characteristic assumptions, $t_3(G)$, $t_2(G)$ and $t(G)$ have the same prime factors (Lemma~\ref{lem.same-prime-factors}), and $t_2(G) = t_3(G)$ (Proposition~\ref{prop.torsion-a}).

\begin{lemma} \label{lem.same-prime-factors}
Let $G$ be a reductive $k$-group, and suppose that $\Char(k)$ is not a torsion prime for $G$. Then the integers $t_3(G)$, $t_2(G)$ and $t(G)$ have the same prime factors.
\end{lemma}

\begin{proof} 
In view of~\eqref{e.divisibilities}, it suffices to show that every prime $p$ dividing $t(G)$ (i.e., every torsion prime of $G$) also divides $t_3(G)$. Indeed, if $p$ is a torsion prime of $G$,
then by~\cite[Theorem 2.28]{steinberg-torsion}, there exists a non-toral elementary abelian $p$-subgroup $(\bZ/p\bZ)^r \simeq A \subset G$ for some $r \geqslant 1$. (In fact, we can take $r \leqslant 3$, though we will not need this.) Since $A$ is non-toral, $p$ divides $\depth(A)$. On the other hand, $\depth(A)$ divides $t_3(G)$ by the definition of $t_3$.
\end{proof}

\begin{proposition} \label{prop.torsion-a}  Let $G$ be a smooth affine $k$-group. Assume that $\Char(k)$ is good for $G$ in the sense of Definition~\ref{def.good_char}. Then $t_2(G) = t_3(G)$.
\end{proposition}

Our proof of Proposition~\ref{prop.torsion-a} will rely on the following lemma. This lemma may be viewed as a partial converse to Theorem~\ref{thm.main}; it will be used again later in the paper.

\begin{lemma} \label{lem.index_equals_depth}
    Let $G$ be an algebraic $k$-group, let $A\subset G$ be a finite abelian subgroup such that $\Char(k)\not\mid |A|$, let $r\geqslant 0$ be an integer, let $E$ be a connected $A$-torsor over $k_r$, and let $E*_AG$ be the $G$-torsor over $k_r$ induced by $E$.
    
    \smallskip
    (1) For every torus $T\subset G$, the $G$-torsor $E*_AG$ is split by a Galois extension $L/k_r$ such that $\Gal(L/k_r) \cong A(k)/A_T(k)$.

    \smallskip
    (2) 
$\ind(E*_AG) = \depth(A)$.
        \end{lemma}

\begin{proof}
(1)  Let $\vphi\colon \Gal(k_r)\to A(k)$ be a continuous homomorphism such that $[E] = [\vphi]$ in $H^1(k_r,A)$. Since $E$ is connected, it is the spectrum of a Galois field extension $K/k_r$. The kernel of $\vphi$ is the absolute Galois group $\Gal(K)\subset \Gal(k_r)$ and $\vphi$ defines an isomorphism $\overline{\vphi}\colon\Gal(K/k_r) \xrightarrow{\sim} A(k)$. Let $L\subset K$ be the fixed field of $\overline{\vphi}^{\; -1}(A_T(k))$. By the Galois correspondence, 
$\vphi^{-1}(A_T(k))$ is precisely $\Gal(L)$ and  $L/k_r$ is Galois with Galois group $A(k)/A_T(k)$. The torsor $E_L$ admits reduction of structure to $A_T$ because $\vphi(\Gal(L))= A_T(k)$. We conclude that $(E*_A G)_L$ admits reduction of structure to $A_T$ and therefore to $T$. By Hilbert's Theorem 90, we have $H^1(L,T) = 0$, and hence $(E*_AG)_L$ is split.

\smallskip
(2) By (1), for every maximal torus $T\subset G$ there exists a splitting field $L_T$ of $E*_AG$ such that $[L_T:k_r] = [A:A_T]$. Therefore $\ind(E*_AG)$ divides $[A:A_T]$ for every torus $T\subset G$, and hence \[ \ind(E*_AG)\mid \depth(A). \]

Conversely, for every splitting field $L$ of $E*_AG$, by Theorem~\ref{thm.main}(1) there exists a torus $T_L \subset G$ such that $[A:A_{T_L}]$ divides $[L:k_r]$. Since $\depth(A)$ divides $[A:A_{T_L}]$, it divides $[L:k_r]$. As this holds for every splitting field $L/k_r$ of $E*_AG$, we get \[ \depth(A)\mid \ind(E*_AG). \] 
We conclude that $\ind(E*_AG)=\depth(A)$, as desired.
\end{proof}

\begin{proof}[Proof of Proposition~\ref{prop.torsion-a}]
    Let $r\geqslant 0$ be an integer. By Lemma~\ref{lem.serre_quest}, any $G$-torsor over $k_r$ is induced from a connected $A$-torsor for some finite abelian subgroup $A\subset G$ of order not divisible by $\Char(k)$.  
    Therefore, Lemma~\ref{lem.index_equals_depth}(2) tells us that
    $$\Big\{ \ind(X) : r\geqslant 0\, , X\in H^1(k_r,G)\Big\} = \Big\{ \depth(A) : A\subset G \text{ finite abelian, }\Char(k)\not\mid|A|\Big\}.$$
    This implies $t_2(G) = t_3(G)$.
\end{proof}

\section{Proofs of Theorems~\ref{thm.t(G)} and~\ref{thm.serre_quest}} 
\label{sect.t(G)}

Recall that an algebraic $k$-group $G$ is called {\em special} if $H^1(K, G) = 1$ for every field extension $K/k$. This notion was introduced by Serre in \cite{serre-special} (reprinted in \cite{serre-reprinted}), where he also showed that special $k$-groups are smooth, affine and connected. Special semisimple groups over the algebraically closed field $k$ were classified by Grothendieck; see~\cite[Theorem 3]{grothendieck-special}.
The notion of special group is of interest over an arbitrary field $k$; see~\cite{huruguen-special, reichstein-tossici-special, merkurjev-special}. Our
standing assumption in this paper is that the base field $k$ is algebraically closed.

\begin{lemma}\label{lem.special_kernel_torality}
    Let $G$ be an algebraic $k$-group, let $N\subset G$ be a normal $k$-subgroup, and let $A\subset G$ be a finite abelian $k$-subgroup such that $\Char(k)\not\mid |A|$. Assume that $N$ is special. Then $A$ is toral in $G$ if and only if the image of $A$ under the homomorphism $\pi\colon G\to G/N$ is toral.
\end{lemma}
\begin{proof}
   If $A$ is contained in a torus $T\subset G$, then $\pi(A)$ is contained in $\pi(T)$, which is a torus in $G/N$; see \cite[Theorem 12.9]{milne2017algebraic}.  Conversely, suppose $\pi(A)$ is toral in $G/N$. Our goal is to show that $A$ is toral in $G$. Let $K/k$ be any field extension and consider the commutative diagram of pointed sets
\begin{equation}\label{diag.special_kernel_toral}
    \begin{tikzcd}
           {\{\ast\}}=H^1(K,N)\arrow[r] & H^1(K,G)\arrow[r]& H^1(K,G/N)\\
            & H^1(K,A) \arrow[u]\arrow[r] & H^1(K,\pi(A)) \arrow[u]
    \end{tikzcd}
\end{equation}
    The map $H^1(K,\pi(A))\to H^1(K,G/N)$ is trivial because $\pi(A)$ is toral. Since the top row of \eqref{diag.special_kernel_toral} is exact, we conclude that the map $H^1(K,A)\to H^1(K,G)$ is trivial.  Since $K/k$ was arbitrary, we conclude that $A$ is toral using Proposition~\ref{prop.special_case}.
\end{proof}

The following proposition is a key ingredient in the proofs of both Theorem~\ref{thm.t(G)} and Theorem~\ref{thm.serre_quest}.

\begin{proposition} \label{prop1} Let $G$ be an algebraic $k$-group, let $m\geqslant 2$ be an integer, and let $A = A_1 \times \dots\times A_m$ be a finite abelian $k$-subgroup of $G$. Assume that $\Char(k) \! \! \not \! | \, |A|$, and $|A_1|,\dots,|A_m|$ are pairwise coprime. If $A_1,\dots,A_m$ are toral, then so is $A$. 
\end{proposition}

\begin{proof}We start by showing that we may assume $G$ is reductive without any loss of generality. Let $H = \Glin$ be the smooth linear part of $G$ (see Section \ref{subsect.notation}).
By definition, $H$ is the maximal smooth connected affine subgroup of $G$. Therefore 
any torus $T\subset G$ is contained in $H$. 
It follows that the toral subgroups $A_1,\dots,A_m$ are contained in $H$. The unipotent radical $R_u(H) \subset H$ is special because $k$ is algebraically closed \cite[Proposition 1]{grothendieck-special}. Therefore, by Lemma~\ref{lem.special_kernel_torality}, the subgroups $A_1,\dots,A_m$ and $A$ are toral in $H$ if and only if their images under the map $\pi\colon H\to H/R_u(H)$ are toral. Since $H/R_u(H)$ is reductive \cite[Proposition 19.11]{milne2017algebraic}, we can replace $A_1,\dots,A_m$ by $\pi(A_1),\dots,\pi(A_m)$ and $G$ by $H/R_u(H)$ to assume $G$ is reductive.

It remains to prove the proposition under the additional assumption that $G$ is reductive. By induction on $m$, we may assume that $m=2$. Thus
$A=A_1\times A_2$, where $|A_1|$ and $|A_2|$ are coprime. Let
$i\colon A\hookrightarrow G$ be the inclusion homomorphism. In view of
Proposition~\ref{prop.special_case}, it suffices to show that, for every
field extension $K/k$, the induced map
$i_*\colon H^1(K,A)\to H^1(K,G)$ is trivial.

Let $\alpha=(\alpha_1,\alpha_2)\in H^1(K,A_1)\times H^1(K,A_2)=H^1(K,A)$.
We will show that $i_*(\alpha)$ is trivial. Since $A_1$ is toral, choose a maximal torus $T_1$ of $G$ containing
$A_1$. The connected centralizer $C_G(A_1)^0$ is reductive (see, e.g., \cite[Corollary 2.18(a)]{steinberg-torsion}) and $T_1$ is a
maximal torus of $C_G(A_1)^0$. Moreover, $A_2$ acts on $C_G(A_1)^0$ by
conjugation, because $A_2$ centralizes $A_1$. By
\cite[Theorem 5.16]{springer-steinberg}, there is a maximal torus $T$ of
$C_G(A_1)^0$ which is preserved by this action.\footnote{In \cite{springer-steinberg}, $G$ is assumed to be semisimple. $C_G(A_1)^0$ is reductive but may not be semisimple in general. We can reduce to the semisimple case by modding out by the center. For an alternative argument using $w$-restricted Lie algebras, see \cite[Proposition 5.10]{chernousov2025loop}.} Since $A_1$ is central in
$C_G(A_1)^0$, it is contained in $T$. Thus $A_1\subset T$ and
$A_2\subset N_G(T)$. Set $N=N_G(T)$, and let $j\colon A\hookrightarrow N$
be the resulting inclusion.

It suffices to show
\begin{equation}\label{eq:claim-j}
j_*(\alpha_1,\alpha_2)=j_*(0,\alpha_2)
\quad\text{in } H^1(K,N).
\end{equation}
Indeed, \eqref{eq:claim-j} implies 
$i_*(\alpha_1,\alpha_2)=i_*(0,\alpha_2)$ in $H^1(K,G)$. But $(0,\alpha_2)$ comes from
$H^1(K,A_2)$, and $A_2$ is toral. Hence Proposition~\ref{prop.special_case}
implies that $i_*(0,\alpha_2)$ is trivial. Thus $i_*(\alpha)$ is trivial, as desired.

It remains to prove \eqref{eq:claim-j}. Set $\beta=(0,\alpha_2)$. Twisting by
a cocycle representing $\beta$ shows that \eqref{eq:claim-j} is equivalent to saying that $(\alpha_1,0)$ maps to the neutral element of $H^1(K,{}_{\alpha_2}N)$ under
\[
({}_{\beta}j)_*\colon
H^1(K,A_1\times {}_{\alpha_2}A_2)\longrightarrow H^1(K,{}_{\alpha_2}N).
\]
Since $(\alpha_1,0)$ comes from the subgroup $A_1$, and since
$A_1\subset T$, its image in $H^1(K,{}_{\alpha_2}N)$ comes from a class
$\gamma\in H^1(K,{}_{\alpha_2}T)$. Here the inclusion
$A_1\hookrightarrow {}_{\alpha_2}T$ is obtained by twisting the inclusion
$A_1\hookrightarrow T$ by $\alpha_2$.

We now show that $\gamma=0$. Since $\alpha_1\in H^1(K,A_1)$, there is a finite extension $L_1/K$ of degree dividing $|A_1|$
which splits $\alpha_1$, and hence splits $\gamma$. On the other
hand, there is a finite extension $L_2/K$ of degree dividing $|A_2|$ which
splits $\alpha_2$. Over $L_2$, the torus ${}_{\alpha_2}T$ becomes split,
and so by Hilbert's Theorem 90 the extension $L_2$ splits $\gamma$. Thus $\gamma$ is split by two finite extensions of $K$ of relatively prime degrees. By \cite[Lemma 3.1]{gordon-sarney}, Serre's question for tori \cite[Question 2]{serre-pp} (i.e., Question~\ref{Totaro_question} with $X$ a $G$-torsor and $d = 1$) has a positive answer. It follows that $\gamma=0$. This proves
\eqref{eq:claim-j}, and hence Proposition~\ref{prop1}.
\end{proof}

Recall that the depth of a finite abelian $k$-subgroup $A\subset G$ is the greatest common divisor of the indices $[A:A_T]$ where $A_T = A\cap T$ ranges over all maximal tori $T\subset G$.

\begin{corollary}\label{cor.of_prop1}
    Let $G$ be an algebraic $k$-group, and let $A \subset G$ be a finite abelian $k$-subgroup such that $\Char(k)\not\mid |A|$. Then there exists a maximal torus $T\subset G$ such that $[A:A_T] = \depth(A)$.
\end{corollary}

\begin{proof}
    Let $p_1,\dots,p_m$ be the prime divisors of $|A|$. Write $A$ as a direct product $A = A_1 \times\dots\times A_m$, where $A_i \subset A$ is a $p_i$-Sylow subgroup for each $1\leq i\leq m$. There exist integers $s_1,\dots,s_m\geqslant 0$ such that 
$$\depth(A)=  p_1^{s_1} p_2^{s_2} \ldots p_m^{s_m}.$$ 
By the definition of $\depth(A)$, for any $1\leq i\leq m$ there exists a torus $T_i\subset G$ such that $p_i^{s_i+1}$ does not divide $[A:A_{T_i}]$. The intersection $A'_i  = A_{T_i} \cap A_i$ is a toral subgroup of $A_i$ such that $[A_i:A'_i] \mid p_i^{s_i}$. The subgroup $A' = A_1'\times\dots\times A'_m$ has index $[A:A']$ dividing $\depth(A)$ by construction. By Proposition~\ref{prop1}, there exists a torus $T \subset G$ such that $A' \subset A_T$. Therefore $$[A:A_T] \mid [A:A'] \mid \depth(A).$$
Since $\depth(A)$ divides $[A:A_T]$ by definition, we conclude $\depth(A) = [A:A_T]$.
\end{proof}

\begin{proof}[Proof of Theorem~\ref{thm.t(G)}]
Let $r\geqslant 0$ be the rank of $A$. By \eqref{eq:gal-kr-presentation}, there exists a surjection $\vphi\colon \mathrm{Gal}(k_r)\to A(k)$. The torsor $[\vphi]=E \in H^1(k_r,A)$ is connected by Lemma~\ref{lem.tame_hom_is_loop}. 
By Lemma~\ref{lem.index_equals_depth}(2), $\depth(A) = \ind(E *_A G)$, where
$E*_A G$  is the $G$-torsor over $k_r$ from the statement of Theorem~\ref{thm.main}.

Combining this with Corollary~\ref{cor.of_prop1},
we conclude that there exists a maximal torus $T \subset G$ such that $[A: A_T] = \ind(E *_A G)$. 
Since $\ind(E *_A G)$ divides $t(G)$ (by the definition of $t(G)$), we conclude that
$[A: A_T]$ divides $t(G)$, as desired.
\end{proof}

\begin{proof}[Proof of Theorem~\ref{thm.serre_quest}]
    Applying Lemma~\ref{lem.serre_quest} to $G$ gives an abelian $k$-subgroup $A \subset G$ and a connected $A$-torsor $E$ such that $X \cong E\ast_A G$.
By Corollary~\ref{cor.of_prop1}, there exists a torus $T\subset G$ such that $[A:A_T] = \depth(A)$. By Lemma~\ref{lem.index_equals_depth}(1) the $G$-torsor $X$ admits a splitting field $L/k_r$ of degree $[L:k_r]=[A:A_T] = \depth(A)$.  Finally, Lemma~\ref{lem.index_equals_depth}(2) implies $[L:k_r]=\depth(A)=\ind(X)$. Since $\ind(X)$ divides $d$, we conclude that $[L:k_r]$ divides $d$, as desired.
\end{proof}

For our final corollary, recall the definition of special group from the beginning of this section.

\begin{corollary} \label{cor.special} Let $G$ be a reductive $k$-group. Assume that $\Char(k)$ is not a torsion prime of $G$. 
Then the following conditions are equivalent:

\smallskip
(a) $G$ is special,

\smallskip
(b) $t(G) = 1$,

\smallskip
(c) $t_2(G) = 1$,

\smallskip
(d) $H^1(k_r, G) = 1$ for every $r \geqslant 0$,

\smallskip
(e) $H^1(k_r, G) = 1$ for every $0\leqslant r \leqslant 3$,

\smallskip
(f) every finite elementary abelian $p$-subgroup of $G$ of rank $\leqslant 3$ is toral for every prime $p \neq \Char(k)$,

\smallskip
(g) every finite abelian subgroup $A$ of $G$ whose order is invertible in $k$ is toral,

\smallskip
(h) $t_3(G) = 1$.
\end{corollary} 

The implication (e) $\Longrightarrow$ (a) refines~\cite[Theorem 1(c)]{reichstein-youssin-special}.

\begin{proof} The overall scheme of the proof is indicated below.
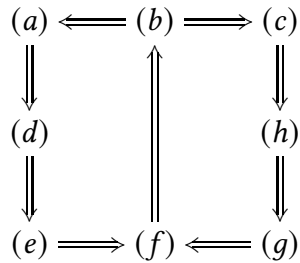
\[
  \xymatrix{
(a) \ar@{=>}[d] \ar@{<=}[r] & (b) \ar@{=>}[r] & (c)   \ar@{=>}[d] \\ 
  (d) \ar@{=>}[d]  &   &  (h) \ar@{=>}[d]   
  \\  (e) \ar@{=>}[r]  & (f) \ar@{=>}[uu]  &   (g) \ar@{=>}[l]    }
\]

\smallskip
The implications (a) $\Longrightarrow$ (d), (b) $\Longrightarrow$ (c),  (d) $\Longrightarrow$ (e), and (g) $\Longrightarrow$ (f) are immediate from the definitions.

\smallskip
(b) $\Longrightarrow$ (a) is proved 
in \cite[Theorem 3]{grothendieck-special}; see also \cite[Corollary 1 to Theorem 2]{grothendieck-special}.

\smallskip
(c) $\Longrightarrow$ (h). As we noted in the Introduction, $t_3(G)$ divides $t_2(G)$ by Theorem~\ref{thm.main}. Thus $t_2(G) = 1$ forces $t_3(G) = 1$.

\smallskip
(d) $\Longrightarrow$ (g) and (e) $\Longrightarrow$ (f) follow from Proposition~\ref{prop.special_case}.

\smallskip
(f) $\Longrightarrow$ (b). Suppose that $p$ is a torsion prime of $G$.
By our assumption, $p \neq \Char(k)$. By \cite[Theorem 2.28]{steinberg-torsion}, $G$ contains a non-toral finite elementary abelian $p$-subgroup $A$ of rank $\leqslant 3$. This contradicts (f). In other words, if (f) holds, then $G$ has no torsion primes. Since the torsion primes of $G$ are precisely the prime divisors of $t(G)$, we conclude that $t(G) = 1$.

\smallskip
(h) $\Longrightarrow$ (g) follows from Corollary~\ref{cor.of_prop1}.
\end{proof}

\section{Proof of Theorem~\ref{thm.torsion-b}}
\label{sect.E8}

Our proof is based on the following lemma.

\begin{lemma}\label{lem.reduction_levis_and_maximals}
Let $G$ be a semisimple $k$-group. For any proper parabolic subgroup $P\subset G$, let $L_P \subset P$ denote a Levi subgroup. The index  $t_3(G)$ is the least common multiple of the depths of all maximal finite abelian subgroups $A\subset G$ and the indices $t_3(L_P)$ as P varies over all the maximal proper parabolic subgroups of $G$.
\end{lemma}

\begin{proof}  
 Let $A\subset G$ be a finite abelian $k$-subgroup of order prime to $\Char(k)$. Then either

    \smallskip
    (i) the centralizer $C_G(A)$ is finite or

    \smallskip
    (ii) $A \subset L_P$, where $P$ is a proper parabolic subgroup of $G$.

\smallskip
\noindent
This follows from~\cite[Lemma 7.6]{reichstein-youssin}; we give
a short proof for the sake of completeness. Assume that (i) fails. 
Then the connected centralizer
$C^0_G(A)$ is reductive and positive-dimensional; see~\cite[Corollary 2.18(a)]{steinberg-torsion}.
Hence, it contains a non-trivial torus $S$.  By \cite[Proposition 20.4]{borel_lag}, $C_G(S) = L_P$ is a Levi subgroup of some parabolic $P \subset G$. By our construction, $A \subset L_P$. Note that $P \neq G$, otherwise $S$ would be central in $G$, contradicting the fact that the center of $G$ is finite. This implies (ii).

Continuing with the proof of Lemma~\ref{lem.reduction_levis_and_maximals},
in case (i), $A$ is contained in a maximal finite abelian subgroup
$A'$. Clearly $\depth(A) \mid \depth(A')$. 

In case (ii), $A \subset L_P$ for some proper parabolic subgroup $P \subsetneq G$. After replacing $P$ by a maximal parabolic $P_{\rm max}$ containing $P$, we may assume that $P$ is maximal. This is because, by \cite[Lemma 2.6]{bate2005geometric}, if $A \subset P_{\max}$, then $A$ lies in some Levi subgroup of $P_{\rm max}$. 
In order to finish the proof of the lemma, it suffices to show that  
\begin{equation} \label{e.parabolic}
\depth(A) = \depth(i(A)), \end{equation}
where $i\colon L_P\hookrightarrow G$ is the natural inclusion map. Here the left hand side denotes the depth of $A$ in $L_P$ and the right hand side denotes the depth of $A$ in $G$. 
If we can establish~\eqref{e.parabolic}, then we will know that the least common multiple of $\depth(i(A))$ over all finite abelian subgroups $A \subset L_P$ such that $|A|$ is prime to $\Char(k)$ is $t_3(L_P)$, and the Lemma will follow.

To prove \eqref{e.parabolic}, let $L/k_r$ be an $A(k)$-Galois field extension, for some $r\geqslant \rank(A)$. Let $E \to \Spec(k_r)$ be the connected $A$-torsor defined by $L/k_r$ and let $[E_{L_P}]\in H^1(k_r,L_{P})$ be the image of $[E]$ under $H^1(k_r,A)\to H^1(k_r,L_P)$. By \cite[Exposé XXVI, Corollaire 5.10(i)]{SGA3}, the pushforward map $i_*\colon H^1(K,L_P) \to H^1(K,G)$ is injective for every field extension $K/k$. 
Therefore $\ind(E_{L_P}) = \ind(i_*(E_{L_P}))$. Applying Lemma~\ref{lem.index_equals_depth}(2) twice yields
    $$\depth(A) = \ind(E_{L_P}) = \ind(i_*(E_{L_P}))= \depth(i(A)),$$
as desired.    
\end{proof}

We now recall the classification of maximal finite abelian subgroups of $E_8$ in characteristic zero due to Draper and Elduque \cite{draper2017maximal}. They showed that, up to conjugacy, there are exactly seven such subgroups: four elementary subgroups, isomorphic to $\mu_2^9,\mu_2^8,\mu_3^5, \mu_5^3$, and three non-elementary subgroups, isomorphic to $\mu_6^3,\mu_4^3\times \mu_2^2, \mu_4\times \mu_2^6$. In particular, up to conjugacy, a maximal finite abelian subgroup $A$ of $E_8$ is uniquely determined by its isomorphism type.

\begin{proposition} \label{prop.torsion-d} Assume $\Char(k)=0$. Let $A$ be a maximal finite abelian subgroup of a split group of type $E_8$. 

\smallskip
(a) If $A \simeq \mu_2^9$, then $\depth(A) = 2$,

\smallskip
(b) If $A \simeq \mu_2^8$, then $\depth(A) = 4$,

\smallskip
(c) If $A \simeq \mu_6^3$, then $\depth(A) \, | \, 6$.

\smallskip
(d) If $A \simeq \mu_5^3$, then $\depth(A) = 5$,

\smallskip
(e) If $A \simeq \mu_4\times \mu_2^6$, then $\depth(A) \, | \, 4$.

\smallskip
(f) If $A \simeq \mu_3^5$, then $\depth(A) = 3$.

\smallskip
(g) If $A \simeq \mu_4^3 \times \mu_2^2$, then $\depth(A) \, | \, 4$.
\end{proposition}

We conclude this section with a proof of Theorem~\ref{thm.torsion-b} based on Proposition~\ref{prop.torsion-d} and Lemma~\ref{lem.reduction_levis_and_maximals}.
In order to preserve the flow of the exposition, 
we defer the proof of Proposition~\ref{prop.torsion-d} to the next section. 

\begin{proof}[Proof of Theorem~\ref{thm.torsion-b}]
Let $G$ be an exceptional group of type $E_8$ over an algebraically closed field of characteristic $0$. Since $t_3(G) = t_2(G)$ by Proposition~\ref{prop.torsion-a}, it suffices to prove
\begin{equation}\label{e.torsion_goal_1}
    t_3(G) = 60.
\end{equation}
The least common multiple of $\depth(A)$ as $A$ varies over all maximal finite abelian subgroups of $G$ is $60$ by Proposition~\ref{prop.torsion-d}. By Lemma~\ref{lem.reduction_levis_and_maximals}, \eqref{e.torsion_goal_1} will follow once we show
\begin{equation}\label{e.torsion_goal_2}
    t_3(L_P) \mid 60,
\end{equation}
for any Levi subgroup $L_P \subset P$ of a maximal parabolic $P\subset G$. 
Since $t_3(L_P)$ divides the Grothendieck torsion index $t(L_P)$, it suffices to show that $t(L_P)$ divides $12$.

Let $L'_P \subset L_P$ be the derived subgroup of $L_P$. The cokernel $L_P/L'_P$ is a torus. By Hilbert 90, $H^1(K,L_P/L_P')=0$ for every field $K$ containing $k$. Thus the pushforward map $i_*\colon H^1(K,L'_P)\to H^1(K,L_P)$ is surjective. Here $i \colon L_P' \hookrightarrow L_P$ is the inclusion map. Since $\ind(i_*(X)) \mid \ind(X)$ for any $[X]\in H^1(K,L'_P)$, we conclude that
$t(L_P)\mid t(L'_P)$. Therefore, in order to establish
\eqref{e.torsion_goal_2}, it suffices to show that
\begin{equation}\label{e.E8_tor_2^5_not_divide_L'_P}
    t(L'_P)\mid 12.
\end{equation}

 Since $P$ is a maximal proper parabolic of $G$, the Dynkin diagram of $L'_P$ is obtained from the Dynkin diagram of $E_8$ by removing one vertex \cite{tits1965classification}. Moreover, $L'_P$ is simply connected by \cite[Corollary 5.4(b)]{springer-steinberg}. Therefore $L'_P$ is a product of simple simply connected groups.
    Removing each of the eight vertices in the Dynkin diagram of $E_8$ gives the following possibilities: 
 \begin{gather*} \text{(i) $D_7$, \, (ii) $A_1\times A_6$, \, (iii) $A_2\times A_1\times A_4$, \, (iv) $A_7$, \, (v) $A_4\times A_3$,} \\
  \text{(vi) $D_5\times A_2$, \,
   (vii) $E_6\times A_1$, \, (viii) $E_7$.}
   \end{gather*}
   
\noindent
   That is, in case (i), $L'_P$ is isomorphic to $\Spin_{14}$, in case (ii) to $\SL_2 \times \SL_7$, etc.
    We will now check that~\eqref{e.E8_tor_2^5_not_divide_L'_P} holds in each of these eight cases.
    By Lemma~\ref{lem.t(G)}, if $H$ is the direct product $H = H_1\times \dots\times H_r$, we have 
    \begin{equation}\label{e.torsion_of_prod}
    \text{$t(H)$ divides $\prod_{i=1,\dots,r} t(H_i)$.}\footnote{Using Grothendieck's definition of the torsion index (see Remark~\ref{rem.chow}), one can show that, in fact, equality holds in \eqref{e.torsion_of_prod}. We will not use this.}
    \end{equation}
    The torsion index of special linear groups is $1$. In cases (ii) - (v), $L'_P$ is a direct product of special linear groups and thus $t(L'_P) =  1$. The simply connected group of type $D_n$ is $\Spin_{2n}$. We have $t(\Spin_{14}) = 2^2$ and $t(\Spin_{10}) = 2$ \cite{totaro-spin}. Therefore $t(L'_P)\mid 4$ in case (i) and $t(L'_P) \mid 2 \cdot 1 = 2$ in case (vi). In cases (vii) and (viii), $t(L'_P)$ divides 
    \[ t(E_6^{\rm sc}) \cdot t(\SL_2) = 6 \cdot 1 = 6 \]
and    
    \[ t(E_7^{\rm sc}) = 12, \]
respectively.
(See~\cite{totaro-E8} for torsion indices of $E_6^{\rm sc}$ and $E_7^{{\rm sc}}$.) We conclude that~\eqref{e.E8_tor_2^5_not_divide_L'_P} holds in each of the cases (i) - (viii), as we wanted to show.
\end{proof}

\begin{remark} Let $G$ be an affine group over an algebraically closed field $k$. Recall that $t(G)$ is defined as the least common multiple of $\ind(X)$, as $X$ ranges over all $G$-torsors $X \to \Spec(K)$, where $K$ is a field containing $k$. By Proposition~\ref{prop:tG-well-defined}, there exists a field $K/k$ and a single (versal) $G$-torsor $X \to \Spec(K)$ such that $\ind(X) = t(G)$. The invariant $t_2(G)$ is defined in a similar manner, as the least common multiple of $\ind(Y)$, over the $G$-torsors $Y \to \Spec(k_r)$. 
It is natural to ask if a similar assertion is 
true for $t_2(G)$. In other words, is there an integer 
$r \geqslant 0$ and a $G$-torsor $Y \to \Spec(k_r)$ such that
$\ind(Y) = t_2(G)$?

For simplicity let us assume that $\Char(k) = 0$. By Lemma~\ref{lem.serre_quest},  
$Y \simeq E*_A G$ for some finite abelian $k$-subgroup $A \subset G$ and some connected $A$-torsor $E \to \Spec(k_r)$. By Lemma~\ref{lem.index_equals_depth}(2), $\depth(A) = \ind(Y)$. 
In other words, our question can be rephrased as follows: is there 
a finite abelian subgroup $A \subset G$ of depth $t_2(G) = t_3(G)$?

Examining the proof of Theorem~\ref{thm.torsion-b},
we see that when $G$ is an exceptional simple group of type $E_8$ and $\Char(k) = 0$ the answer is ``no''. Indeed, in this case, $t_2(G) = 60$ but $\depth(A) \leqslant 12$ for every finite abelian subgroup $A \subset G$.
\end{remark}

\section{Depths of the maximal finite abelian subgroups
of \texorpdfstring{$E_8$}{E8}}
\label{sect.conclusion}

In this section we prove Proposition~\ref{prop.torsion-d}. Let $k$ be a field of characteristic zero, and let $G$ be a semisimple $k$-group of type $E_8$. Recall from \cite{draper2017maximal} that a maximal finite abelian subgroup $A \subset G$ is determined up to conjugacy by its isomorphism type. 

\subsection{Case (a)} 
We may take $A$ to be the subgroup generated by the 2-torsion subgroup $T[2]$ of a maximal torus $T \subset G$ and an element of the normalizer $N(T)$ which acts on $T$ via multiplication by $-1$. Indeed, $A \cong \mu_2^9$ and $C_G(A) = A$; see Theorem 2.17 and the table on p. 258 in~\cite{griess1991elementary}. Therefore, $A$ is maximal. Hence $\depth(A) \, | \, [A: T[2]] = 2$. Since $A$ is not toral, we have equality, $\depth(A) = 2$.

\subsection{Case (b)} 
By \cite[Proposition 5.3]{reichstein-youssin2}, $G$ has a self-centralizing subgroup $A \simeq \mu_2^8$ of depth $4$~\footnote{Note that the definition of depth in \cite{reichstein-youssin2} is slightly different. There the depth of an abelian $p$-subgroup $A \subset G$ is defined as the smallest exponent $e$ such that $A$ has a toral subgroup of index $p^e$. In this paper we say that the depth is $p^e$, rather than $e$. The advantage of the definition of depth in this paper is that it works equally well if $A$ is not a $p$-group.}. Since all maximal finite abelian subgroups of $G$ isomorphic to $\mu_2^8$ are conjugate, Proposition~\ref{prop.torsion-d}(b) follows. 

\subsection{Cases (c) and (d)}
By \cite[Corollary 5.25(c)]{steinberg-torsion} any 
finite abelian subgroup of $G$ of rank $2$ (i.e., generated by $2$ elements) is toral. In particular, 
this tells us that in Case (c) $\mu_6^2 \subset A$ is toral, and thus $\depth(A) \, | \, 6$. In Case (d), $\mu_5^2$ is toral and hence, $\depth(A) \, | \, 5$. Since $A$ is not toral, i.e., 
$\depth(A) \neq 1$, we conclude that $\depth(A) = 5$.

\subsection{Case (e)} We use the correspondence between finite abelian subgroups of spin groups and self-orthogonal binary linear codes introduced in \cite{wood}. Let $\mathbb F_2$ be the field of two elements. Recall that a \emph{linear code of length $n$} is an $\mathbb F_2$-linear subspace $C \subset \bF_2^n$. Elements of $C$ are called \emph{words}. We denote the $i$-th coordinate of a word $w\in C$ by $w_i$. The \emph{Hamming weight} of $w\in \bF_2^n$ is the number of non-zero coordinates
$$\wt(w) = \big{|} \{ 1\leq i\leq n : w_i=1\} \big{|}.$$
We use the notation $(\bF_2^n)_0 \subset \bF_2^n$ for the subspace consisting of all even weight words.
 We regard  $\bF_2^n$ as an inner product space with respect to the standard dot product
$$u\cdot v = \sum_{i=1}^n u_i v_i$$
taking values in $\mathbb F_2$. A code $C\subset \bF_2^n$ is called \emph{self-orthogonal} if $C \subset C^{\perp}$ and \emph{self-dual} if $C = C^{\perp}$.

Let $\Spin_{2n}$ be the spin group associated to the quadratic space $(V,q)$, where $V = k^{2n}$ and  $q(x) = x_1 ^2+\dots +x_{2n}^2$. Let $e_1,\dots,e_{2n}$ be the standard basis of $V$. 
For any vector $v\in (\bF_2^{2n})_0$ let \[e_v\coloneqq \prod_{i=1}^ne_i^{v_i}\in \Spin_{2n}(k),\] where the product is taken in the Clifford algebra of $(V,q)$. For any two vectors $u,v\in (\bF_2^{2n})_0$, we have 
\begin{equation}\label{e.prop_of_e_u_e_v} 
e_v e_u = (-1)^{u \cdot v} e_u e_v \ \text{ and } \ e_v^2 = (-1)^{\wt(v)/2}.
\end{equation}
In particular, any self-orthogonal code $C \subset (\bF_2^{2n})_0$ defines a finite \'etale abelian $k$-subgroup $A_{C}\subset \Spin_{2n}$ by the formula
$$A_C(k) = \langle e_v : v\in C\rangle\subset \Spin_{2n}(k).$$
The following lemma collects the properties of $A_C$ needed for the proof of Proposition~\ref{prop.torsion-d}(e).

\begin{lemma}\label{lem.depth_in_Spin}
    Let $C\subset (\bF_2^{2n})_0$ be a self-orthogonal code with basis $v^1,\dots,v^m \in (\bF_2^{2n})_0$. 
    
   \smallskip
   (1) If there exists $w\in C$ such that $\wt(w)\equiv 2\mod 4$, then $A_C \cong \mu_4\times \mu_2^{m-1}$.
    
   \smallskip
   (2) If $C$ is self-dual and $\wt(w)\neq 2$ for all $w\in C$, then $\mathfrak {so}_{2n}^{A_C} = 0$.   

   \smallskip
   (3) Assume $\{1,\dots,2n\} = \{x_1,y_1\}\cup\dots\cup \{x_{n},y_n\}$ for some numbers $x_1,\dots,y_n$. If we have
   \begin{equation}\label{e.torality_in_spin}
    \text{for any }1\leqslant i\leqslant m\text{ and  }1\leqslant j\leqslant n: \ \ v^i_{x_j} = 1 \iff v^i_{y_j}=1,
\end{equation}
   then the subgroup $A_C$ is toral.
\end{lemma}
\begin{proof}
    (1) One checks that the function $$\eps\colon C\to \mu_2,\quad w\mapsto (-1)^{\wt(w)/2}$$ is a homomorphism using \eqref{e.prop_of_e_u_e_v}. The kernel of $\eps$ consists of all $w\in C$ such that $\wt(w) \equiv 0\mod 4$. By assumption, $\eps$ is surjective, so $\dim_{\bF_2}(\ker\eps) = m-1$. Therefore, up to a change of basis, we may assume that $\wt(v^1)\equiv 2 \mod 4$ and $\wt(v^2),\dots,\wt(v^m)$ are divisible by $4$.  Then $e_{v^1}$ is of order $4$ and $e_{v^2},\dots,e_{v^m}$ are of order $2$ by  \eqref{e.prop_of_e_u_e_v}. Therefore $A_C\cong \mu_4\times \mu_2^{m-1}$.

    \smallskip
    (2) Let $E_{i,j}$ ($1\leqslant i,j\leqslant 2n$) be the standard basis of $M_{2n}(k)$. The Lie algebra $\mathfrak{so}_{2n}$ of skew-symmetric $2n \times 2n$-matrices  decomposes as a direct sum of $1$-dimensional $A_C$-eigenspaces \begin{equation}\label{e.eigen_decomposition_so_n}
    \mathfrak{so}_{2n} = \bigoplus_{1\leqslant i<j\leqslant 2n} \Span_k (E_{ij}- E_{ji}).
    \end{equation}
    For any word $w\in C$, we compute
    $$e_{w}(E_{ij}- E_{ji})e_w^{-1}  = (-1)^{w_i +w_j}(E_{ij}- E_{ji}).$$
    If $E_{ij}- E_{ji}$ is fixed by $A_C$, then $w_i+w_j = 0$ for all $w\in C$. In other words, we have $u^{ij} \in C^{\perp}$, where $u^{ij}\in \bF_2^{2n}$ is the vector whose only non-zero coordinates are the $i$-th and $j$-th coordinates. However, $u^{ij} \not\in C^{\perp}$ because $\wt(u^{ij}) = 2$ and $C = C^{\perp}$ by assumption. Therefore $A_C$ does not act trivially on any of the eigenspaces $\Span_k (E_{ij}-E_{ji})$. We conclude that $\mathfrak{so}_{2n}^{A_C} = 0$ from \eqref{e.eigen_decomposition_so_n}. 

    \smallskip
    (3) Up to relabeling the coordinates, we may assume $x_j = 2j-1$ and $y_j = 2j$ for all $1\leqslant j\leqslant n$. We denote the standard block-diagonal torus of $\SO_{2n}$ by 
$$ T = \begin{pmatrix}
     \SO_2 & 0 & 0 & \ldots & 0 \\
     0 & \SO_2 & 0 & \ldots & 0 \\
     \vdots & \vdots & \vdots & \ldots & \vdots \\
       0 & 0 & 0 & \ldots & \SO_2 
 \end{pmatrix} \simeq \SO_2^n \simeq \mathbb G_m^n.$$   
(Recall that $\SO_2 \simeq \mathbb G_m$ is a $1$-dimensional torus). Let $\pi\colon \Spin_{2n}\to \SO_{2n}$ be the covering map described in \cite[Page 279]{wood}. For all $1\leqslant i\leqslant m$, we have $$\pi(e_{v^i}) = \diag((-1)^{v^i_1},\dots,(-1)^{v^i_{2n}}).$$
    Therefore \eqref{e.torality_in_spin} implies $\pi(e_{v^i}) \in T$ for all $i$. In particular, $A_C$ is contained in the subgroup $\pi^{-1}(T)\subset \Spin_{2n}$, which is a torus by \cite[Proposition 11.14]{borel_lag}.
\end{proof}

\begin{proof}[Proof of Proposition~\ref{prop.torsion-d}(e)]
 Let $C \subset (\bF_2^{16})_0$ be the code generated by the rows of the following $8$-by-$16$ matrix
 \setcounter{MaxMatrixCols}{20}
 $$\begin{pmatrix}
     1 & 1 &1 &1& 0 & 0 & 0 & 0 & 0 & 0 & 0 & 0 & 0 & 0 & 0 & 0 \\
     0&0& 1 & 1 &1 &1&  0 & 0 & 0 & 0 & 0 & 0 & 0 & 0 & 0 & 0 \\
     0 & 0 & 0 & 0 & 1 & 1 &1 &1&  0 & 0 & 0 & 0 & 0 & 0 & 0 & 0 \\
     0 & 0 & 0 & 0 & 0 & 0 & 0 & 0 & 1 & 1 &1 &1& 0 & 0 & 0 & 0\\
      0 & 0 & 0 & 0 & 0 & 0 & 0 & 0 &0&0& 1 & 1 &1 &1&  0 & 0 \\
      0 & 0 & 0 & 0 & 0 & 0 & 0 & 0 & 0 & 0 & 0 & 0 & 1 & 1 &1 &1 \\
       1 & 0 & 1 & 0 & 1 & 0 & 1 & 0 & 0 & 0 & 0 & 0 & 0 & 0 &1 &1 \\
        1 & 1 & 0 & 0 & 0 & 0 & 0 & 0 & 0 & 1 & 0 & 1 & 0 & 1 &0 &1 
 \end{pmatrix}$$
 Then $C$ is an $8$-dimensional self-dual code containing no words of weight $2$; see \cite[Table 2]{pless1972classification} where this code is denoted $F_{16}$. Let $C_{2,8} \subset C$ be the subcode generated by all rows of the above matrix except for the second and eighth rows. Since $C$ and $C_{2,8}$ contain words of weight $6$, Lemma~\ref{lem.depth_in_Spin}(1) implies
 $$ A_C \cong \mu_4 \times \mu_2^7, \ \text{ and }\ A_{C_{2,8}} \cong \mu_4\times \mu_2^5.$$
 The rows generating $C_{2,8}$ satisfy \eqref{e.torality_in_spin} with respect to the partition $$\{1,\dots,16\}=\{1,3\}\cup\{2,4\}\cup\{5,7\}\cup\{6,8\}\cup\{9,10\}\cup\{11,12\}\cup\{13,14\}\cup\{15,16\}.$$ 
 Therefore $A_{C_{2,8}}$ is toral by Lemma~\ref{lem.depth_in_Spin}(3). Since $A_{C_{2,8}}\subset A_C$, we have
 $$\depth(A_C) \mid [A_C:A_{C_{2,8}}] = 4.$$
 Recall that $\HSpin_{16}$ is the quotient $ \Spin_{16}/\langle e_{\mathbf{1}}\rangle,$ where $\mathbf{1}\in (\bF_2^{16})_0$ is the vector with all coordinates equal to one.
There exists a morphism $\Phi\colon \Spin_{16} \to G$ such that $\ker\Phi = \langle e_{\mathbf{1}}\rangle$ and  $\HSpin_{16} \cong \im\Phi=C_G(\Phi(-1))$; see \cite[Example 4.3]{garibaldi2016E8}. Consider the subgroup
$$A  = \Phi(A_C) \subset G.$$
Then $\depth(A) \mid \depth(A_C) \mid 4$, so it suffices to prove:

\smallskip (1) $A\cong \mu_4\times \mu_2^6$  and

\smallskip (2) $A \subset G$ is a maximal finite abelian subgroup. 

\smallskip\noindent
Since the vector $\mathbf{1} \in (\bF_2^{16})_0$ is orthogonal to $C$ and $C$ is self-dual, we have $\mathbf{1}\in C$. The quotient $A_C/\langle e_{\mathbf{1}}\rangle$ contains an element of order $4$ because $e_w^2 \neq e_{\mathbf{1}}$ for any $w\in C$. This proves (1) because 
$$A \cong A_C/\langle e_{\mathbf{1}}\rangle \cong \mu_4\times\mu_2^6.$$
Since $\im\Phi=C_G(\Phi(-1))$, $\Phi$ induces an isomorphism
$$\mathfrak e_8^{\Phi(-1)} \cong \mathfrak{so}_{16},$$
such that $\Spin_{16}$ acts on $\mathfrak e_8^{\Phi(-1)} \cong \mathfrak{so}_{16}$ via the adjoint representation. The element $\Phi(-1)$ is in $A = \Phi(A_C)$ because $e_w^{2} = -1$ for any weight-$6$ word $w\in C$. Therefore 
$$\mathfrak e_8^{A} = (\mathfrak e_8^{\Phi(-1)} )^{A} \cong \mathfrak{so}_{16}^{A_C} = \{0\},$$
where the last equality follows from Lemma~\ref{lem.depth_in_Spin}(2). Since $\mathfrak e_8^{A}$ is the Lie algebra of $C_G(A)$, we conclude that $C_G(A)$ is finite. Therefore $A \subset \tilde{A}$ for some maximal finite abelian subgroup $\tilde{A} \subset G$. Referring to \cite[Theorem 1.1]{draper2017maximal}, we see that any maximal finite abelian subgroup of $G$ containing $\mu_4\times \mu_2^6$ must be isomorphic to $\mu_4\times \mu_2^6$. Therefore $A = \tilde{A}$ is maximal. This shows that (2) holds, finishing the proof.  
\end{proof}

\subsection{Case (f)} The following lemma will be used in the proofs of parts (f) and (g) of Proposition~\ref{prop.torsion-d}.

\begin{lemma} \label{lem.lie-algebra} Let $H$ be a semisimple
algebraic group over a field $k$ of characteristic $0$. Let  
$A$ be a finite abelian subgroup of $H$, and let $A_0$ be a subgroup of $A$. Assume that 

\begin{itemize}
\item
the centralizer $C_H(A)$ is finite,

\item the quotient $A/A_0$ is cyclic, and

\item $\dim_k(\mathfrak h^{A_0}) \geqslant \rank(H)$. Here $A_0$ acts on the Lie algebra $\mathfrak h$ of $H$ via the adjoint action,
and $\mathfrak h^{A_0}$ denotes the subalgebra of elements fixed by $A_0$. 
\end{itemize}

Then $A_0$ is toral in $H$. 
\end{lemma}

\begin{proof} Since $A/A_0$ is cyclic, we can write $A = \langle A_0, \theta \rangle$ for some $\theta \in A$.
    
    We claim that $\mathfrak h^{A_0}$ is abelian. Indeed, assume the contrary: $\mathfrak s \coloneqq  [\mathfrak h^{A_0},\mathfrak h^{A_0}] \neq 0$. 
    Since $\mathfrak h^{A_0}$ is reductive, $\mathfrak s$ is semisimple.  Moreover, $\mathfrak s$ is $\theta$-invariant because $A_0$ commutes with $\theta$.
    By a theorem of Borel and Mostow~\cite[Proposition~4.3]{borel1955semi-simple}, 
    \[ 0 \neq {\mathfrak s}^\theta \subset (\mathfrak h^{A_0})^\theta = \mathfrak h^A, \]
    On the other hand, by our assumption $C_H(A)$ is finite and hence, its tangent space $\mathfrak h^A = 0$. This contradiction shows that $\mathfrak s=0$, and hence proves the claim. 

      Since $A_0$ is diagonalizable, its connected centralizer $C_H(A_0)^\circ$ is reductive. Moreover, by the above claim, its Lie algebra $\mathfrak h^{A_0}$ is commutative. Consequently, $C_H(A_0)^\circ$ is reductive and commutative, i.e., $C_H(A_0)^\circ$ is a torus. By our assumption, 
       \begin{equation} \label{e.maximal-torus}
       \dim  C_H(A_0) = \dim (\mathfrak h^{A_0}) \geqslant \rank(H) \, . 
       \end{equation}
        Since $\rank(H)$ is, by definition, the dimension of a maximal torus in $H$, we conclude that the inequality in~\eqref{e.maximal-torus} is an equality, and
        $T = C_H(A_0)^\circ$ is a maximal torus of $H$. Therefore
        $A_0 \subset C_H(T) = T$, and the lemma follows.
\end{proof}

\begin{proof}[Proof of Proposition~\ref{prop.torsion-d}(f)] 
The adjoint action
of the maximal finite abelian subgroup $A \simeq \mu_3^5$ decomposes the 248-dimensional Lie algebra $\mathfrak g$ of $G$ as a direct sum of character spaces
\[ \mathfrak g_{\chi_0}, \mathfrak  g_{-\chi_0}, \mathfrak g_{\chi_1}, \mathfrak g_{\chi_2}, 
\ldots ,\mathfrak g_{\chi_{240}}, \]
where $\chi_0$, $- \chi_0$, and $\chi_i$, $i = 1, 2, \ldots, 240$ are the $242$ non-trivial characters of $A$,
$\dim(\mathfrak g_{\chi_0}) = \dim (\mathfrak g_{-\chi_0}) = 4$ and $\dim(\mathfrak g_{\chi_i}) = 1$ for every $i = 1, 2, \ldots, 240$. For details, see Proposition 6.30(2) in \cite{elduque-kochetov13} and the paragraph preceding this proposition.

Let $A_0$ be the kernel of the character
$\chi_0 \colon A \to \mu_3$. Then $\mathfrak{g}^{A_0} = \mathfrak g_{\chi_0} \oplus \mathfrak g_{-\chi_0}$
is an $8$-dimensional $k$-vector space. By Lemma~\ref{lem.lie-algebra},
$A_0$ is toral. Hence, $\depth(A) \, | \, [A:A_0] = 3$. Since $\depth(A) \neq 1$ (i.e., $A$ is not toral), we conclude that $\depth(A) = 3$, as desired.
\end{proof}

\subsection{Case (g)}
We identify $G$ with $\Aut(\mathfrak g)$, where $\mathfrak g =\mathfrak e_8$ is the Lie algebra of $G$. We recall the setup of \cite[Section 5, p.~8]{draper2017maximal}. Let $\theta$ be an order $4$ automorphism of $\mathfrak g$ of type I. Then $\theta$ induces a $\mathbb Z/4\mathbb Z$-grading
\[
        \mathfrak g
        =
        \mathfrak g_{0}\oplus
        \mathfrak g_{1}\oplus
        \mathfrak g_{2}\oplus
        \mathfrak g_{3},
        \qquad
        \mathfrak g_{r}
        =
        \{x\in\mathfrak g\mid \theta(x)=i^r x\} \quad(0\leq r\leq 3),
\]
where $i$ is a primitive fourth root of unity. We have
\begin{equation}
\label{eq:DE-type-I}
\begin{aligned}
        \mathfrak g_{0} &= \mathfrak{sl}(U)\oplus \mathfrak{sl}(V),\\
        \mathfrak g_{1} &= U\otimes \wedge^2 V,\\
        \mathfrak g_{2} &= \wedge^4 V,\\
        \mathfrak g_{3} &= U\otimes \wedge^6 V,
\end{aligned}
\end{equation}
where $U$ and $V$ are $k$-vector spaces of dimensions $2$ and $8$, respectively. We also have homomorphisms
\[
        \Phi\colon \SL(U)\times \SL(V)\longrightarrow C_G(\theta),
        \qquad
        (a,b)\longmapsto \phi_{a,b},
\]
and
\[
        \Psi\colon C_G(\theta)\longrightarrow \Aut(\mathfrak g_0),
        \qquad
        \varphi\longmapsto \varphi|_{\mathfrak g_0},
\]
and moreover for all $(a,b)\in \SL(U)\times \SL(V)$
\[
        \phi_{a,b}|_{\mathfrak g_0}=(\Ad a,\Ad b),
        \qquad
        \phi_{a,b}|_{\mathfrak g_1}=a\otimes \wedge^2 b.
\]
In particular, $\theta=\phi_{1_U,\zeta_81_V}$, where $\zeta_8^2=i$. By \cite[Lemma 5.1]{draper2017maximal}, the homomorphism $\Phi$ is surjective with kernel $\langle(-1_U,i1_V)\rangle\simeq \mathbb Z/4\mathbb Z$. 

Define
\begin{equation}
\label{eq.DE_Q_type_I}
        Q=
        \left\langle
        \phi_{a_1,b_1},
        \phi_{a_2,b_2},
        \phi_{1_U,c_1},
        \phi_{1_U,c_2},
        \phi_{1_U,\zeta_81_V}
        \right\rangle\subset G,
\end{equation}
where $a_1,a_2\in \SL(U)$ and $b_1,b_2,c_1,c_2\in \SL(V)$ are chosen such that
\begin{equation}
\label{eq.DE_type_I_relations}
\begin{gathered}
        a_i^2=-1_U,\qquad b_i^4=-1_V,\qquad c_i^2=-1_V
        \quad (i=1,2),\\
        b_1b_2=ib_2b_1,\qquad
        c_1c_2=-c_2c_1,\qquad
        b_i c_j=c_j b_i.
\end{gathered}
\end{equation}
By \cite[Theorem 5.2]{draper2017maximal}, up to replacing $A$ by a conjugate $k$-subgroup, we may assume that $A=Q$. Again by \cite[Theorem 5.2]{draper2017maximal}, we have $Q\simeq \mu_4^3\times \mu_2^2$.

For every subset $S\subset \End(V)$, let $\alg\langle S\rangle\subset \End(V)$ be the associative $k$-subalgebra of $\End(V)$ generated by $S$. From \cite[Proof of Theorem 5.2]{draper2017maximal}\footnote{More specifically, see the last three paragraphs of the proof.}, 
we may write
\[
        \End(V)=\alg\langle b_1,b_2\rangle\otimes \alg\langle c_1,c_2\rangle,
        \qquad
        \alg\langle b_1,b_2\rangle\simeq M_4(k),\qquad \alg\langle c_1,c_2\rangle\simeq M_2(k).
\]
Define
\[
        A_0 =
        \left\langle
        \phi_{a_1,b_1},
        \phi_{a_2,b_2},
        \phi_{1_U,c_1},
        \phi_{1_U,c_2}
        \right\rangle
        \subset Q.
\]
Recalling that $\theta=\phi_{1_U,\zeta_81_V}$, we see that $Q=\langle A_0,\theta\rangle$. We have $[Q: A_0]=4$. Therefore, in order to complete the proof
of Proposition~\ref{prop.torsion-d}(g), it suffices to show that $A_0$ is toral. In view of Lemma~\ref{lem.lie-algebra}, it is enough to establish the following

\begin{claim}\label{claim:dim-gB-8}
    We have $\dim(\mathfrak g^{A_0})=8$.
\end{claim} 

We now proceed with the proof of Claim~\ref{claim:dim-gB-8}.
As $A_0$ commutes with $\theta$, the decomposition \eqref{eq:DE-type-I} is $A_0$-stable,
and hence
\[
        \mathfrak g^{A_0}
        =
        \mathfrak g_{0}^{A_0}\oplus
        \mathfrak g_{1}^{A_0}\oplus
        \mathfrak g_{2}^{A_0}\oplus
        \mathfrak g_{3}^{A_0}.
\]
By \cite[Section 2, p.~3]{draper2017maximal}, if $Q$ is a finite
maximal quasi-torus of $\Aut(\mathfrak g)$, then the neutral homogeneous
component of the corresponding fine grading is trivial; equivalently, $\mathfrak g^Q=0$ \footnote{Alternatively, one may check that $\mathfrak g^Q= (\mathfrak{sl}(U)\oplus \mathfrak{sl}(V))^Q = 0$ using $\alg\langle a_1,a_2\rangle = \End(U), \alg\langle b_1,b_2\rangle = \End(V_1)$, and $ \alg\langle c_1,c_2\rangle = \End(V_2)$. This gives another way of verifying that $Q$ is a maximal finite abelian subgroup. Indeed, $\mathfrak g^Q =0$ implies $C_G(Q)$ is finite and therefore $Q$ is contained in some maximal finite abelian subgroup $\tilde{Q}$. Since $Q\cong \mu_4^3\times \mu_2^2$, \cite[Theorem 1.1]{draper2017maximal} implies $\tilde{Q}$ has to be equal to $Q$.}. 
In our situation $Q=\langle A_0,\theta\rangle$, and $A_0$ commutes with
$\theta$. Since $\mathfrak g_0=\mathfrak g^\theta$, we get
\begin{equation}\label{eq:g0-inv}
\mathfrak g_0^{A_0} = (\mathfrak g^\theta)^{A_0} = (\mathfrak g^{A_0})^\theta = \mathfrak g^{\langle {A_0},\theta\rangle} = \mathfrak g^Q = 0.
\end{equation}
Following the proof of
\cite[Theorem 5.2, p.~9]{draper2017maximal}, let $\zeta\in k$ be such that $\zeta^2=i$. We may choose
\[
        a_1=\begin{pmatrix}
        i&0\\
        0&-i
        \end{pmatrix},\qquad
        a_2=
        \begin{pmatrix}
        0&1\\
        -1&0
        \end{pmatrix},
\]
and
\[
        b_1=\zeta P\otimes 1_{V_2},
        \qquad
        b_2=\zeta S\otimes 1_{V_2},
        \qquad
        c_1=1_{V_1}\otimes
        \begin{pmatrix}
        i&0\\
        0&-i
        \end{pmatrix},
        \qquad
        c_2=1_{V_1}\otimes
        \begin{pmatrix}
        0&1\\
        -1&0
        \end{pmatrix},
\]
where
\[
        P=\operatorname{diag}(1,i,-1,-i),
        \qquad
        S(e_j)=e_{j+1}
        \quad (j\in \mathbb Z/4\mathbb Z),
\]
for some basis $e_0,e_1,e_2,e_3$ of $V_1$. One easily checks that \eqref{eq.DE_type_I_relations} holds. 

For $r,s\in\{0,1,2,3\}$ and $e,f\in\{0,1\}$, set
\[
        g_{r,s,e,f}
        =
        \phi_{a_1,b_1}^r
        \phi_{a_2,b_2}^s
        \phi_{1_U,c_1}^e
        \phi_{1_U,c_2}^f
        =
        \phi_{a_1^r a_2^s,\,
        b_1^r b_2^s c_1^e c_2^f}.
\]
Since $\Char(k)=0$, for every finite-dimensional ${A_0}$-module $M$,
\begin{equation}\label{eq:tr-dim}
        \dim M^{A_0}
        =
        \frac{1}{64}
        \sum
        \operatorname{tr}(g_{r,s,e,f}\mid M).
\end{equation}
For an endomorphism $T$ of $V$, we use the identity
\[
        \operatorname{tr}(\wedge^m T) = [x^m]\det(1+xT),
\]
where $[x^m]$ denotes the coefficient of $x^m$. Therefore,
\begin{equation}\label{eq:tr1}
        \operatorname{tr}(g_{r,s,e,f}\mid U\otimes\wedge^m V)=\operatorname{tr}(a_1^r a_2^s)\cdot [x^m]\det(1+xb_1^r b_2^s c_1^e c_2^f ),
\end{equation}
whereas
\begin{equation}\label{eq:tr2}
        \operatorname{tr}(g_{r,s,e,f}\mid \wedge^4 V)
        =
        [x^4]\det(
        1+xb_1^r b_2^s c_1^e c_2^f
        ).
\end{equation}

Let $B_{r,s}\coloneqq b_1^r b_2^s|_{V_1}=\zeta^{r+s}P^rS^s$.
Since 
$c_1^e c_2^f$ acts on $V_2$ with eigenvalues $1,1$ for $(e,f)=(0,0)$ and with eigenvalues $i,-i$ for the other three choices of $(e,f)$, we have
\[
\begin{aligned}
        \sum_{e,f}
        \det(1+x\,b_1^r b_2^s c_1^e c_2^f)
        &=
        \det(1+xB_{r,s})^2
        +
        3\det(1+x^2B_{r,s}^2).
\end{aligned}
\]
Set $F_{r,s}(x)\coloneqq\det(1+xB_{r,s})^2+3\det(1+x^2B_{r,s}^2)$, so that the previous equation implies
\[
       \sum_{e,f}
        [x^m]\det(1+x\,b_1^r b_2^s c_1^e c_2^f)=[x^m]F_{r,s}(x).
\]

The characteristic polynomials of the monomial matrices
$B_{r,s}=\zeta^{r+s}P^rS^s$ give the following table:
\[
\begin{array}{c|c|c|c|c|c}
(r,s)&\det(1+xB_{r,s})&\det(1+x^2B_{r,s}^2)&[x^2]F_{r,s}&[x^4]F_{r,s}&[x^6]F_{r,s}\\
\hline
(0,0)&(1+x)^4&(1+x^2)^4&40&88&40\\
(2,0),(0,2)&(1+x^2)^2&(1-x^2)^4&-8&24&-8\\
(2,2)&(1-x^2)^2&(1+x^2)^4&8&24&8\\
\text{all other }(r,s)
&1+x^4&(1+x^4)^2&0&8&0.
\end{array}
\]
Moreover, one immediately sees
\[
        \operatorname{tr}(a_1^r a_2^s)
        =
        \begin{cases}
        2, & (r,s)=(0,0),\\
        -2, & (r,s)=(2,0)\text{ or }(0,2),\\
        2, & (r,s)=(2,2),\\
        0, & \text{otherwise}.
        \end{cases}
\]
Therefore, \eqref{eq:tr1} yields
\[
\sum_{r,s,e,f}\operatorname{tr}(g_{r,s,e,f}\mid U\otimes\wedge^2V)=
2\cdot 40+(-2)(-8)+(-2)(-8)+2\cdot 8=128\]
and
\[
\sum_{r,s,e,f}\operatorname{tr}(g_{r,s,e,f}\mid U\otimes\wedge^6V)=2\cdot 40+(-2)(-8)+(-2)(-8)+2\cdot 8=128,\]
while \eqref{eq:tr2} gives
\[
\sum_{r,s,e,f}\operatorname{tr}(g_{r,s,e,f}\mid \wedge^4V)=
88+24+24+24+12\cdot 8=256.\]
By \eqref{eq:tr-dim}, we deduce that
\[
        \dim \mathfrak g_{1}^{A_0}=\frac{128}{64}=2,\qquad
        \dim \mathfrak g_{2}^{A_0}=\frac{256}{64}=4,\qquad
        \dim \mathfrak g_{3}^{A_0}=\frac{128}{64}=2,
\]
and hence $\dim \mathfrak g^{A_0}=0+2+4+2=8$, as claimed. This completes the proof of Claim~\ref{claim:dim-gB-8} and thus of Proposition~\ref{prop.torsion-d}(g).

\section{Proof of Theorem~\ref{thm.non-toral-obstruction}}
\label{sect.genus1}
   
\begin{definition}	For a prime $p$ and a finite group $\Gamma$, the $p$-rank of $\Gamma$, denoted by $\on{rank}_p(\Gamma)$, is the largest integer $r\geqslant 0$ such that $\Gamma$ has a subgroup isomorphic to $(\mathbb{Z}/p\mathbb{Z})^r$.
\end{definition}

\begin{proof}[Proof of Theorem~\ref{thm.non-toral-obstruction}]
We argue by contradiction. Suppose that there exist a $g$-dimensional abelian variety $B$ over $k_r$ and a $B$-torsor $X$ such that $E*_AG$ becomes split over $k_r(X)$. 

By \cite[Lemma 6.1]{genus1}, there exists a finite extension $K/k_r$ of prime-to-$p$ degree such that the class of $X_{K}$ in the abelian group $H^1(K,B)$ has order $p^e$ for some $e\geqslant 0$. Since $k$ is algebraically closed of characteristic not equal to $p$, it contains a root of unity of order $p^e$. By \cite[Proposition 3.1]{genus1}, there exists a finite Galois extension $L/K$ such that $X(L)\neq \emptyset$ and such that $\Gamma\coloneqq \operatorname{Gal}(L/K)$ fits into an exact sequence of finite groups
\[
        1\longrightarrow N\longrightarrow \Gamma\longrightarrow \Gamma_0
        \longrightarrow 1,
\]
where $N\subset (\mathbb Z/p^e\mathbb Z)^{2g}$ and $\Gamma_0\subset \operatorname{SL}_{2g}(\mathbb Z/p^e\mathbb Z)$. By the subadditivity of the $p$-rank in short exact sequences
		\[\on{rank}_p(\on{Gal}(L/K)) \leqslant \on{rank}_p(\mathrm{SL}_{2g}(\mathbb{Z}/p^e\mathbb{Z}))+\on{rank}_p((\mathbb{Z}/p^e\mathbb{Z})^{2g}),\]
		and hence by \cite[Corollary 4.4]{genus1} we have
		\begin{equation}\label{eq:p-rank-upper-bound}
		\mathrm{rank}_p(\on{Gal}(L/K)) \leqslant
		\begin{cases}
			5g^2 +2g-1 & (p>2,\, g\geqslant 2), \\
			9g^2 + 2g-2 & (p = 2,\, g\geqslant 2), \\
			5 & (p>2,\, g=1), \\
			6 & (p=2,\, g=1).
		\end{cases}
		\end{equation}
Since $E*_AG$ is split by $X$ and $X(L)\neq \emptyset$, the pullback of $E*_AG$ to $\operatorname{Spec}(L)$ is trivial. Thus $L$ is a splitting field of $E*_AG$ over $K$. The extension $K/k_r$ has degree prime to $|A|=p^r$, and $L/K$ is Galois. Therefore, by Theorem~\ref{thm.main}(2) applied to the tower $k_r\subset K\subset L$, there exists a maximal torus $T\subset G$ such that $A/A_T$ is isomorphic to a subgroup of $\Gamma$, where $A_T=A\cap T$. Now the lower bound on $\mathrm{rank}(A/A_T)$ in the statement of Theorem~\ref{thm.non-toral-obstruction} contradicts \eqref{eq:p-rank-upper-bound}.
\end{proof}

Now recall from the beginning of Section~\ref{sect.t(G)} that an algebraic group $G$ defined over $k$ is called special if $H^1(K, G) = 1$ for every field extension $K/k$. We characterized special groups in Corollary~\ref{cor.special}.
As a consequence of Theorem~\ref{thm.non-toral-obstruction}, we obtain the following alternative characterization.

\begin{corollary}\label{cor:special-splitting-genus-one}
   Let $G$ be a reductive $k$-group such that $\Char(k)$ is not a torsion prime of $G$. The following are equivalent:
    \begin{itemize}
        \item[(1)] $G$ is not special;
                \item[(2)] there exist a field extension $K/k$ and seven $G$-torsors $E_1,\dots,E_7$ over $K$ that cannot all be simultaneously split by any genus $1$ curve over $K$.
    \end{itemize}
\end{corollary}

\begin{proof}
(2)$\Longrightarrow$(1). This is clear.

\smallskip

(1)$\Longrightarrow$(2). By Corollary~\ref{cor.special}, there exists a prime $p\neq \operatorname{char}(k)$ and a non-toral elementary abelian $p$-subgroup $A\subset G$ of rank $r\leqslant 3$. Every maximal torus $T$ of $G^7$ is of the form $T = T_1 \times \ldots \times T_7$ for some maximal tori $T_1, \ldots, T_7$ of $G$; in particular, $\mathrm{rank}(A^7/(A^7)_T)\geqslant 7$.

By Theorem~\ref{thm.non-toral-obstruction} applied to the algebraic $k$-group $G^7$ and the abelian subgroup $A^7\subset G^7$ of rank $7r$, there exists a $G^7$-torsor $E$ over $k_{7r}$ which is not split by a genus $1$ curve over $k_{7r}$. Let $E_1,\dots,E_7$ be $G$-torsors over $k_{7r}$ such that $E\cong E_1\times\cdots\times E_7$. If $C/k_{7r}$ is a genus $1$ curve splitting $E_1,\dots,E_7$, then $C$ splits $E$, a contradiction.
\end{proof}

\begin{remark}
    It is not true that a reductive $k$-group $G$ is special if and only if, for every field extension $K/k$, every $G$-torsor over $K$ is split by a genus $1$ curve over $K$. The smallest counterexample is $G=\PGL_2$; see \cite{swets}.
\end{remark}

\begin{remark}
    Let $G$ be a reductive $k$-group, let $p\neq\Char(k)$ be a torsion prime for $G$, and let $A\subset G$ be an elementary abelian $p$-subgroup of $G$. The non-toral $p$-rank of $A$ is the smallest rank of $A/(A\cap T)$, where $T$ ranges over the maximal tori of $G$. It is equal to the depth of $A$ in the sense of \cite[Definition~4.5]{reichstein-youssin2}.

    Table~\ref{tab:RY-elementary-lower-bounds} records some examples of simple $k$-groups $G$ with elementary abelian $p$-subgroups of non-toral $p$-rank $\geqslant 2$.
 \end{remark}
 
\begin{table}[ht]
\centering
\renewcommand{\arraystretch}{1.25}
\begin{tabular}{c|c|c|c}
\hline
$G$ & $p$ & $\exists\, \mu_p^r\simeq A\subset G$ of non-toral $p$-rank $\geqslant$ & Reference \\
\hline
$\PGL_{p^m}$ & $p$ &
$m$ &
\cite[Corollary~7.9]{reichstein-youssin2} \\
\hline
$\SO_n$ & $2$ &
$\left\lfloor (n-1)/2\right\rfloor$ &
\cite[Proof of Proposition~5.2]{reichstein-youssin2} \\
\hline
$E_7^{\mathrm{ad}}$ & $2$ &
$2$ &
\cite[Proposition~5.7]{reichstein-youssin2} \\
\hline
$E_8$ & $2$ &
$2$ &
\cite[Proposition~5.3]{reichstein-youssin2} \\
\hline
\end{tabular}
\caption{Some elementary abelian $p$-subgroups of non-toral $p$-rank $\geqslant 2$.}
\label{tab:RY-elementary-lower-bounds}
\end{table}

 \begin{remark} 
 Suppose that $\Char(k)\neq 2$. In the case, where $G = \SO_n$, Theorem~\ref{thm.non-toral-obstruction} and Table~\ref{tab:RY-elementary-lower-bounds} tells us that for every $n \geqslant 15$ there exist a field extension $K/k$ and a quadratic form $q$ of rank $n$ over $K$ such that $q$ has trivial discriminant and $q$ does not become hyperbolic over the function field of any genus $1$ curve over $K$. 
 
 For $G = E_8$, arguing as in the proof of Corollary~\ref{cor:special-splitting-genus-one}, we see that there exists a field $K/k$ and four $E_8$-torsors over $K$ which cannot be simultaneously split by any genus $1$ curve over $K$.
\end{remark}

\section*{Acknowledgements} Theorem~\ref{thm.t(G)} was conjectured by Burt Totaro and Jiahe Wang. We are grateful to them for sharing this conjecture with us. We also thank Burt Totaro for helpful comments on preliminary drafts of this paper. 


\end{document}